\documentclass[12pt]{amsart}
\usepackage{amsmath,amssymb,hyperref}
\usepackage{graphics}
\usepackage[all]{xy}

\newtheorem{thm}{Theorem}
\newtheorem*{theo}{Theorem}
\newtheorem{cor}[thm]{Corollary}
\newtheorem{prop}[thm]{Proposition}
\newtheorem{lem}[thm]{Lemma}

\newtheorem{remark}[thm]{Remark}
\newtheorem{example}[thm]{Example}

\newcommand{\OM}[1]{\Omega^1_{#1}} 
\newcommand{\OO}[1]{\mathcal O_{#1}}

\DeclareMathOperator{\elm}{elm}

\DeclareMathOperator{\tr}{tr}

\def\Z{\mathbb Z}

\def\C{\mathbb C}
\def\P{\mathbb P}

\begin{document}

\title[Lam\'e connections and Okamoto symetry of $P_{VI}$]{Isomonodromic deformation of Lam\'e connections, 
Painlev\'e VI equation and Okamoto symetry}
\author[Frank Loray]{Frank LORAY}
\address{Frank LORAY (Directeur de Recherches au CNRS)
\hfill\break IRMAR, Campus de Beaulieu, 35042 Rennes Cedex (France)}
\email{frank.loray@univ-rennes1.fr}
\date{\today}
\dedicatory{In the memory of Andrei Bolibrukh}

\begin{abstract} A Lam\'e connection is a logarithmic $\mathrm{sl}(2,\C)$-connection $(E,\nabla)$ over an elliptic curve 
$X:\{y^2=x(x-1)(x-t)\}$, $t\not=0,1$, having a single pole at infinity. 
When this connection is irreducible, we show that
it is invariant by the standart involution and
can be pushed down as a logarithmic $\mathrm{sl}(2,\C)$-connection over $\P^1$
with poles at $0$, $1$, $t$ and $\infty$.
Therefore, the isomonodromic deformation $(E_t,\nabla_t)$
of an irreducible Lam\'e connection, when the elliptic
curve $X_t$ varry in the Legendre family, is parametrized
by a solution $q(t)$ of the Painlev\'e VI differential equation $P_{VI}$. We compute
the variation of the underlying vector bundle $E_t$ along
the deformation via Tu moduli map: it is given by another solution
$\tilde q(t)$ of $P_{VI}$ equation related to $q(t)$ by the Okamoto symetry $s_2 s_1 s_2$ (Noumi-Yamada notation). 
Motivated by the Riemann-Hilbert problem for the classical
Lam\'e equation, the question whether Painlev\'e transcendents 
do have poles is raised.
Those results were partially announced in \cite{Kyoto}.
\end{abstract}

\maketitle
\tableofcontents

\section*{Motivations}
The classical Lam\'e equation (here in Legendre form)
\begin{equation}\label{E:LameEquationLegendre}
\frac{d^2u}{dx^2}+{1\over 2}\left(\frac{1}{x}+\frac{1}{x-1}+\frac{1}{x-t}\right)\frac{du}{dx}+\frac{c-n(n+1)x}{4x(x-1)(x-t)}u=0,
\end{equation}
$t,n,c\in\C$, $t\not=0,1$, 
defines a projective structure on the elliptic curve 
$$X:\{y^2=x(x-1)(x-t)\},\ \ \ t\in\C-\{0,1\}$$
having a fuchsian singularity at the point $\omega_\infty$ at infinity.
On the affine part of the curve, projective charts
are local determinations of 
$$\phi:X-\{\omega_\infty\}\to\P^1\ ;\ (x,y)\mapsto \frac{u_1(x)}{u_2(x)}$$
where $u_1$ and $u_2$ range over independant solutions of 
(\ref{E:LameEquationLegendre}). Setting $\vartheta:=2n$,
there is a local coordinate $z$ at $\omega_\infty$ such that
one of the projective charts around this point is given by
\begin{itemize}
\item $\phi=z^\vartheta$ when $\vartheta\not\in\Z$,
\item $\phi=z^m$ or $\phi=\frac{1}{z^m}+\log(z)$ 
when $\vartheta=\pm m$, $m\in\Z_{>0}$, 
\item $\phi=z$ (regular) when $\vartheta=0$.
\end{itemize}
The monodromy of any projective chart $\phi=\frac{u_1}{u_2}$ after analytic continuation
along any loop $\gamma$ is given by $\rho(\gamma)\circ \phi$
where 
\begin{equation}\label{EqMonRep}
\rho:\pi_1(X-\{\omega_\infty\})\to\mathrm{PGL}(2,\C)
\end{equation}
is the projective monodromy representation of (\ref{E:LameEquationLegendre}) computed in the basis $(u_1,u_2)$.
The following natural question goes back to Poincar\'e (see \cite{Poincare}):
{\it which topological representation
\begin{equation}\label{EqTopRep}
\rho:\pi_1(\text{once punctured torus})\to\mathrm{PGL}(2,\C)
\end{equation}
occurs as the monodromy of a fuchsian projective structure ?}

Any fuchsian projective structure on the once punctured torus
(i.e. having moderate growth at the puncture) is of the form above:
the parameter $t$ stands for the underlying complex structure,
$n$ (or $\vartheta=2n+1$) for the fuchsian type of the puncture
and $c$ is the accessory parameter.
The number of parameters 
fits with the dimension of the space of those
topological representations up to conjugacy (see \cite{Goldman}).
One thus expects that a generic representation should
be the monodromy of some Lam\'e equation.
The similar question for regular projective structures 
on complete curves have been answered only recently 
in \cite{GalloKapovichMarden}
by pants decomposition and gluing methods. 
Our initial aim, originating the present work, was to 
use isomonodromy method in order to answer the
Lam\'e case as a test; as we shall see, we actually reduce
the above question to the existence of poles for Painlev\'e VI
transcendents, which looks difficult, 
of a different nature though.

The Lam\'e equation may be viewed as a logarithmic $\mathrm{sl}(2,\C)$-connection 
on the trivial vector bundle $\mathcal O\oplus\mathcal O$
over the elliptic curve $X$, having a single pole at $\omega_\infty$:
any eigenvector of the residual matrix connection at $\omega_\infty$
provides a cyclic vector going back to the scalar elliptic form.
On the other hand, the Riemann-Hilbert correspondence asserts
that any representation 
$$\rho:\pi_1(X-\{\omega_\infty\})\to\mathrm{SL}(2,\C)$$
is the monodromy of a logarithmic $\mathrm{sl}(2,\C)$-connection $\nabla$
on some rank $2$ vector bundle $E$ over $X$, 
having a single pole at $\omega_\infty$.
Our initial question becomes: {\it for a given topological representation (\ref{EqTopRep}),
can we choose the complex structure $X$ in such a way that
the realizing connection $(E,\nabla)$ is defined on the trivial bundle ?}
Now, the question perfectly fits into the setting of isomonodromic deformations.

Starting with a ``Lam\'e connection'' $(E_0,\nabla_0)$ on $X_{t_0}$, consider its isomonodromic deformation
$(E_t,\nabla_t)$ when the complex structure of the curve $X_t$
varry; here, the parameter $t$ of the deformation has to be considered
in the universal cover $T\simeq\mathbb H\to\P^1-\{0,1,\infty\}$,
the Teichm\"uller space of the once punctured torus. 
When the pole of the Lam\'e connection is not apparent, this 
unique deformation is characterized by the fact that its monodromy
representation is locally constant. Equivalently, the deformation 
is induced by the
unique logarithmic integrable connection on the
universal curve $\boldsymbol X\to T$ having a single pole
along the section $t\mapsto\omega_\infty\in X_t$ 
whose restriction to the initial fiber $X_{t_0}$ 
coincides with $(E_0,\nabla_0)$.
Our initial question becomes: {\it for which connection/representation
can we insure that the underlying vector bundle $E_t$ becomes
trivial for some convenient parameter $t$ along the isomonodromic
deformation of a Lam\'e connection ?} 

In this paper, we will compute the variation of the vector bundle $E_t$
along the isomonodromic deformation of a Lam\'e connection and
show that it is given by a solution of the Painlev\'e VI equation 
for some convenient parameter; the bundle $E_t$ becomes then
trivial only when the Painlev\'e transcendent has a pole (outside $0$, 
$1$ and $\infty$). We are finally
led to the following question which seems widely open: 
{\it do Painlev\'e transcendents have poles ?} 

Actually, much more interesting is the similar question for regular connections
over genus $g>1$ curves. Following \cite{Gunning}, regular projective structures
correspond there to regular connections on the maximally unstable
undecomposable $\mathrm{SL}(2,\C)$-bundle 
(a $\mathrm{sl}(2,\C)$-oper in the sense of 
\cite{BeilinsonDrinfeld}, \S 2.7). The main result
of \cite{GalloKapovichMarden} can be rephrased as follows:
{\it the special bundle above occur along the isomonodromic
deformation if, and only if, the monodromy is neither reducible, 
nor in $\mathrm{SU}(2,\C)$ (up to $\mathrm{SL}(2,\C)$-conjugacy)}. 
Can we prove this directly by computing the variation of the bundle ? 

Another interseting question when $g>1$ is whether a given topological representation 
$\pi_1(X^{top})\to\mathrm{SL}(2,\C)$ can be realized as the monodromy of a connection on the trivial bundle,
for a convenient choice of complex structure $X$ ?
In the case the image $\Gamma$ is a discrete subgroup, this provides an embedding
$X\to \mathrm{SL}(2,\C)/\Gamma$: the fundamental matrix of the associate system
defines an equivariant map $\tilde X\to \mathrm{SL}(2,\C)$ (where $\tilde X\to X$ the universal cover).
The existence of compact curves in quotients of $\mathrm{SL}(2,\C)$ is still an open problem.
The isomonodromic approach gather into the same geometrical framework questions arising from various contexts.

\section{Result}

Isomonodromic deformations of meromorphic connections
have been extensively studied over the Riemann sphere
(see \cite{Malgrange,Palmer}). In this situation, the underlying
vector bundle is constant on a Zariski open subset of the parameter
space (see \cite{Bolibrukh}) allowing an explicit computation 
of the isomonodromy condition, namely the Schlesinger equations;
Painlev\'e equations arise after further reduction in the simplest
case: rank $2$ with $4$ poles. In order to observe
some continuous deformation of the underlying bundle,
one has to switch to connections over genus $g>0$ curves.
The simplest non trivial case, namely regular rank $1$ connections
over an elliptic curve, has been considered in \cite{Hitchin,Lame}:
there, it was observed that the variation of the underlying line bundle
along an isomonodromic deformation is a Painlev\'e transcendent.
In this paper, we study the next case in difficulty, 
namely rank $2$ logarithmic connections over elliptic curves with a single pole.

Throughout this work, a {\bf Lam\'e connection} consists 
in the data $(E,\nabla)$ of a rank $2$ locally trivial holomorphic 
vector bundle $E$
over an elliptic curve belonging to the Legendre family:
\begin{equation}\label{Eq:Legendre}
X:\{y^2=x(x-1)(x-t)\},\ \ \ t\in\C-\{0,1\}
\end{equation}
equipped with a logarithmic and trace free connection $\nabla$
having (at most) a single pole at the point $\omega_\infty$ at infinity:
$$\nabla:E\to E\otimes\OM{X}([\omega_\infty]),\ \ \ 
\det(E)=\OO{X},\ \ \ \mathrm{trace}(\nabla)=d:\OO{X}\to\OM{X}$$
(here, we identify vector bundles $E$ and $\Omega$ with the 
corresponding sheaves of holomorphic sections). 
Up to now, such a connection $(E,\nabla)$ will be considered 
up to holomorphic bundle isomorphism. The {\bf exponent} $\vartheta\in\C$ of $(E,\nabla)$,
defined up to a sign, is the difference between the 
eigenvalues $\pm\frac{\vartheta}{2}$ of the residual matrix 
connection at $\omega_\infty$.

The underlying vector bundle of a Lam\'e connection has trivial
determinant because of the trace free condition.
Following Atiyah \cite{Atiyah}, almost all rank $2$ vector bundles on $X$ 
with trivial determinant are decomposable, i.e. of the form 
\begin{equation}\label{E:DecomposableBundle}
E=L\oplus L^{-1}\ \ \ \ \text{with}\ \ \ L\in\text{Pic}(X);
\end{equation}
to complete the list, one has to add $4$ extra bundles $E_i$, $i=0,1,t,\infty$. 
Semistable bundles are those decomposable ones with $L\in\text{Pic}^0(X)$, i.e. of the form
\begin{equation}\label{E:SemistableBundle}
E=L\oplus L^{-1}\ \ \ \ \text{with}\ \ \ L=\OO{X}([\omega]-[\omega_\infty]),\ \omega=(x,y)\in X,
\end{equation}
together with the $4$ undecomposable ones above. The corresponding
moduli space is $\P^1$ (see \cite{Tu}) with quotient map
given by
\begin{equation}\label{E:ModuliSemistableBundle}
\lambda\ :\  \left\{\begin{matrix}E\mapsto x&\text{under notation of (\ref{E:SemistableBundle})}\\ E_i\mapsto i&i=0,1,t,\infty \end{matrix}\right.
\end{equation}
If we denote by $\omega_i=(i,0)\in X$ the 
$2$-torsion points, $i=0,1,t,\infty$, then $E_i$ is in fact the unique non trivial extension
\begin{equation}\label{E:UndecomposableBundles}
0\to L_i\to E_i\to L_i\to 0,\ \ \ \text{with}\ 
L_i=\mathcal O([\omega_i]-[\omega_\infty]);
\end{equation}
the moduli map $\lambda$ identifies $E_i$ with the trivial extension $L_i\oplus L_i$. In particular, the point $\lambda=\infty$
corresponds to both the trivial vector bundle and $E_\infty$.

The {\bf isomonodromic deformation} of a Lam\'e connection is defined 
as follows. Consider the universal cover
$T\to\P^1-\{0,1,\infty\}$, namely the Teichm\"uller
space of the punctured torus, and the universal Legendre curve
$\boldsymbol X\to T$ over this parameter space: the fiber $X_t$ at some point $t$ 
is precisely the curve $\{y^2=x(x-1)(x-t)\}$ (by abuse of notation, we denote by $t$
a point of $T$ and its projection on $\P^1$).
The point $\omega_\infty$ at infinity defines a section $T\to\boldsymbol X$ of this fibration.
Starting from a Lam\'e connection $(E_0,\nabla_0)$ 
over $X_{t_0}$, there is a unique logarithmic flat
connection $(E,\nabla)$ over the total space $\boldsymbol X$,
having the section $\omega_\infty$ as polar set
and inducing the initial connection $(E_0,\nabla_0)$
on $X_{t_0}$ (see \cite{Deligne,Malgrange,Heu}). 
The deformation $t\mapsto(E_t,\nabla_t)$
induced over the family $t\mapsto X_t$ is the isomonodromic
deformation of $(E_0,\nabla_0)$. When the pole 
of $\nabla_0$ is not an apparent singular point 
(i.e. having local monodromy $\pm I$), then $t\mapsto(E_t,\nabla_t)$
is precisely the unique deformation having constant monodromy representation.
The exponent $\vartheta$ of the Lam\'e connection is constant along such a deformation. Finally, one can talk about the variation of 
the underlying vector bundle $E_t$ along the deformation 
just by considering the moduli map $t\mapsto \lambda(E_t)$ defined above.
Our result is

\begin{thm}\label{T:Main} Let $(E_t,\nabla_t)$ be the isomonodromic 
deformation of an irreducible Lam\'e connection.
Then, for a Zariski open subset of the parameter space $T$, 
the underlying vector bundle 
$E_t$ is semistable (see \cite{Heu}) and its Tu invariant $t\mapsto\lambda(E_t)\in\P^1$
defined by (\ref{E:ModuliSemistableBundle}) is solution of 
the Painlev\'e VI differential equation
\begin{equation}\label{E:PainleveEquation}
\begin{matrix}{ \frac{d^2 \lambda}{dt^2}=\frac{1}{2}\left(\frac{1}{\lambda}+\frac{1}{\lambda-1}+\frac{1}{\lambda-t}\right)
\left(\frac{d\lambda}{dt}\right)^2
-\left(\frac{1}{t}+\frac{1}{t-1}+\frac{1}{\lambda-t}\right)
\left(\frac{d\lambda}{dt}\right)}\\
\small{+\frac{\lambda(\lambda-1)(\lambda-t)}{t^2(t-1)^2}
\left(\frac{\kappa_\infty^2}{2}-\frac{\kappa_0^2}{2}\frac{t}{\lambda^2}
+\frac{\kappa_1^2}{2}\frac{t-1}{(\lambda-1)^2}+\frac{1-\kappa_t^2}{2}\frac{t(t-1)}{(\lambda-t)^2}\right)}.
\end{matrix}
\end{equation}
with parameter $(\kappa_0,\kappa_1,\kappa_t,\kappa_\infty)
=(\frac{\vartheta}{4},\frac{\vartheta}{4},\frac{\vartheta}{4},\frac{\vartheta}{4})$ where $\vartheta$ is the exponent of the Lam\'e connection. 
\end{thm}

By ``Zariski open'',
we just mean that exceptional values of $t$ form 
a discrete subset of the parameter space $T$. 
This actually directly follows from \cite{Heu}.

It is already known that isomonodromic deformation 
of (generic) Lam\'e connections 
are parametrized by Painlev\'e VI equation with parameters specified above:
in \cite{Kawai,LevinOlshanetsky}, isomonodromic deformation equations
are directly computed on the elliptic curve, and the elliptic form
of Painlev\'e VI equation (see \cite {Manin}) is recognized.

Our approach of this result is quite different. We first prove, 
using the Riemann-Hilbert correspondence, that any irreducible
Lam\'e connection can be pushed down, via the $2$-fold
cover $X\to\P^1$ as a logarithmic
rank $2$ connection over $\P^1$ with $4$ poles,
namely at the ramification values $0$, $1$, $t$ and $\infty$.
We are back to the classical case of Fuchs:
the isomonodromic deformation is parametrized 
by a solution, say $q(t)$,
of the Painlev\'e VI equation with parameter
$(\kappa_0,\kappa_1,\kappa_t,\kappa_\infty)
=(\frac{1}{2},\frac{1}{2},\frac{1}{2},\frac{\vartheta-1}{2})$.
This already explains why Painlev\'e VI equation arises
in the Lam\'e case; so far, no computation is needed. 
Elementary birational geometry is used to go back
to the initial deformation $(E_t,\nabla_t)$ and 
the moduli $\lambda$ of the vector bundle $E_t$ can be 
computed by means of $q(t)$: we recognize in $\lambda(E_t)$ 
the image of $q(t)$
by the Okamoto symetry $s_1s_2s_1$ (see \cite{NoumiYamada}).
This automatically implies that $\lambda(E_t)$ is also solution 
of the Painlev\'e VI equation, but with new parameter 
$(\frac{\vartheta}{4},\frac{\vartheta}{4},\frac{\vartheta}{4},\frac{\vartheta}{4})$. 

A similar statement holds for the 
classical Painlev\'e VI setting (logarithmic $\mathrm{sl}(2,\C)$-connections over $\P^1$ with $4$ poles),
when considering the parabolic bundle defined by eigendirections of the residual matrix of the connection (see \cite{ArinkinLysenko,LoraySaitoSimpson}).

More generally, we can start with the isomonodromic deformation of a rank $2$ logarithmic connection 
with $4$ poles over $\P^1$, parametrized by 
any Painlev\'e VI transcendent $q(t)$. Then one can lift-up conveniently 
the deformation over the Legendre elliptic curve $X_t$ as a rank $2$
logarithmic connection with poles at the ramification points $\omega_i$
in such a way that the moduli $\lambda(E_t)$ of the vector bundle is the Okamoto symetric $s_1s_2s_1$ of $q(t)$; this provides a new geometric interpretation of this strange symetry. We thus
obtain in a natural way an isomonodromic deformation problem
(a Lax pair) for the general elliptic form of the Painlev\'e VI equation,
just by considering those rank $2$ and trace free logarithmic
connection over $X_t$ having poles at the $2$-order points
$\omega_i$ that moreover commute with the elliptic involution
$(x,y)\mapsto(x,-y)$; this has been also considered in \cite{Zotov}.

When we set $\vartheta=0$, all Lam\'e connections with
vanishing exponent are reducible, but  those regular ones 
can still be pushed down to $\P^1$: our result 
remains valid in this case and we retrieve the

\begin{cor}[\cite{Hitchin,Lame}]\label{Cor:Hitchin} Let $t\mapsto(L_t,\nabla_t)$ be the isomonodromic deformation
of a regular rank $1$ connection
on the Legendre deformation $X_t$ and let 
$E_t=\mathcal O\left([\omega(t)]-[\omega_\infty]\right)$ be the
underlying line bundle, $\omega(t)=(x(t),y(t))\in X_t$. Then $t\mapsto x(t)$
is solution of Painlev\'e VI equation
with parameters $(\kappa_0,\kappa_1,\kappa_t,\kappa_\infty)=(0,0,0,0)$.
\end{cor}

One can explicitely compute, in this case, the variation of the line bundle $L_t$ by means of elliptic functions and retrieve the

\begin{cor}[Picard \cite{Picard}, see \cite{Mazzocco}]\label{Cor:Picard}
The general solution of Painlev\'e VI equation
with parameters $(\kappa_0,\kappa_1,\kappa_t,\kappa_\infty)=(0,0,0,0)$
is given by
$$t\mapsto x(t)\ \ \ \text{where}\ \ \ (x(t),y(t)):=\pi(c_0\cdot\omega_0+c_1\cdot\omega_1),\ c_0,c_1\in\C$$
where $\pi:\C\to X_t$ is the universal cover and $\omega_i(t)$, half-periods of $X_t$.
\end{cor}

The Painlev\'e transcendents of Corollary \ref{Cor:Picard}
have poles if, and only if, either $c_0$ or $c_1$ is not real.

\section{Our main construction: elliptic pull-back}\label{sec:MainConstruction}
Here, we construct Lam\'e connections by lifting
on the elliptic two-fold cover $\pi:X\to\P^1$
certain $\mathrm{sl}(2,\C)$-connections 
having logarithmic poles at the critical values of $\pi$.
Later, we will prove that all irreducible Lam\'e
connections can be obtained by this way;
this will be used to parametrize their isomonodromic
deformation by means of Painlev\'e VI solutions in an explicit way.

Let us fix exponents $\boldsymbol{\theta}=(\theta_0,\theta_1,\theta_t,\theta_\infty)$ and consider a logarithmic $\mathrm{sl}(2,\C)$-connection $(E,\nabla)$ over $\P^1$ with poles at $0$, $1$, 
$t$ and $\infty$ and prescribed exponents:
eigenvalues of the residual matrix at $i$ are $\pm\frac{\theta_i}{2}$, 
for $i=0,1,t,\infty$. Such a connection will be called a
{\bf Heun connection}.

The important data that will be used later is
the parabolic structure 
$\boldsymbol{l}=(l_0,l_1,l_t,l_\infty)$ defined 
by the eigenline $l_i\in\P(E\vert_i)$
of the residue of $\nabla$ at $i$ with respect to the eigenvalue 
$-\frac{\theta_i}{2}$, for $i=0,1,t,\infty$. If $\nabla$ do
have pole at each $\omega_i$, the parabolic structure is
perfectly well defined by the connection and the choice of 
exponents $\theta_i$ (they are defined up to a sign).
However, it is important to allow non singular points 
in our construction if we want to fit with the usual Painlev\'e VI 
phase space (see \cite{IIS}):
when $\theta_i=0$ and the corresponding
point is non singular, then any line $l_i\in\P(E\vert_i)$ 
is an eigenline and we have to choose one for our construction.
We call {\bf parabolic Heun connection} with parameter 
$\boldsymbol{\theta}$ the data $(E,\nabla,\boldsymbol{l})$
with above properties.

\begin{example}When $E$ is the trivial bundle, 
then $\nabla$ is defined by a fuchsian system
\begin{equation}\label{E:FuchsianSystem4}
{d Y\over dx}=\left({A_0\over x}+{A_1\over x-1}+{A_t\over x-t}\right)Y,\ \ \ 
A_i\in\text{sl}(2,\C).
\end{equation}
The residual matrix at $x=i$ is given by $A_i$ 
for $i=0,1,t,\infty$ where $A_\infty$ 
is defined by 
\begin{equation}\label{E:ResiduInfini}
A_0+A_1+A_t+A_\infty=0.
\end{equation}
Exponent restrictions are given by $\det(A_i)=-\frac{\theta_i^2}{4}$
and the parabolic structure, by $l_i=\ker(A_i+\frac{\theta_i}{2}I)$.
\end{example}

In order to motivate the following construction, let us indicate
that for special exponents 
$$\boldsymbol{\theta}=(\frac{1}{2},\frac{1}{2},\frac{1}{2},\frac{1}{2}+\frac{\vartheta}{2}),$$
it will provide a Lam\'e connection with exponent $\vartheta$ at infinity.

{\bf Step 1 :} We pull-back the connection $(E,\nabla)$ on the elliptic
cover 
$$\pi:X\to\P^1\ ;\ (x,y)\mapsto x.$$ 
We obtain a logarithmic $\mathrm{sl}(2,\C)$-connection 
on $X$ 
$$(\tilde E,\tilde \nabla):=\pi^*(E,\nabla)$$
with poles at the ramification points 
$\omega_0$, $\omega_1$, $\omega_t$ and $\omega_\infty$
and twice the initial exponents $2\theta_i$. The parabolic structure
$\tilde{\boldsymbol{l}}:=\pi^*\boldsymbol{l}$ corresponds to eigenlines with respect to
eigenvalues $-\theta_i$. 
The point $i$ is parabolic (resp. non singular) for $\nabla$
if, and only if, the point $\omega_i$ is so for $\tilde \nabla$.
We already note that, 
for exponents $\theta_i=\frac{1}{2}$ for $i=0,1,t$, 
the corresponding singular points $\omega_i$
of $(\tilde E,\tilde \nabla)$ are projectively apparent, 
i.e. having $-I$ local monodromy. The next two steps will clear them off.

Remark that we could have choosen an initial connection 
on $\P^1$ with a single pole at $\infty$ so that its 
lifting is of Lam\'e type; but the monodromy would then be 
trivial in this case, and the Lam\'e connection ``very reducible''.

{\bf Step 2 :} We make a convenient birational bundle modification
$$\phi:\tilde E\dashrightarrow\tilde E'$$
such that the new connection 
$$(\tilde E',\tilde \nabla'):=\phi_*(\tilde E,\tilde \nabla)$$
is still logarithmic, having poles at $\omega_i$, with
eigenvalues $\theta_i$ and $1-\theta_i$, for  $i=0,1,t,\infty$. 
This is done by applying successive elementary transformations
$\mathrm{elm}^+$ with respect to the parabolic structure $\tilde{\boldsymbol{l}}$
over each singular point (see section \ref{ss:elm} for the definition and properties of $\mathrm{elm}^+_l$):
$$\phi=\mathrm{elm}^+_{\tilde l_0}\circ\mathrm{elm}^+_{\tilde l_1}\circ\mathrm{elm}^+_{\tilde l_t}\circ\mathrm{elm}^+_{\tilde l_\infty}.
$$
The new connection is no more trace free: we have 
$$\det(\tilde E')=\mathcal O_X([\omega_0]+[\omega_1]+[\omega_t]+[\omega_\infty])\simeq\mathcal O_X(4[\omega_\infty]),$$
and the trace $\tr(\tilde \nabla')$ is the unique logarithmic connection
on $\mathcal O_X(4[\omega_\infty])$ having poles 
at each $\omega_i$ with residue $+1$ and trivial monodromy.
Over the affine chart $X^*$, $\tr(\tilde \nabla')$ is defined in a 
convenient trivialization of the line bundle by
$$d-\left(\frac{dx}{x}+\frac{dx}{x-1}+\frac{dx}{x-t}\right).$$

{\bf Step 3 :} We now twist $(\tilde E,\tilde \nabla)$
by a convenient rank one connection in order to restore the trace-free property. For this, {\bf we choose} the unique square root $(L,\zeta)$
of $(\det(\tilde E'),\tr(\tilde \nabla'))$
defined on the line bundle $L=\mathcal O_X(2[\omega_\infty])$:
$\zeta$ is given over the affine chart by
$$d-\left(\frac{dx}{2x}+\frac{dx}{2(x-1)}+\frac{dx}{2(x-t)}\right).$$
The resulting $\mathrm{sl}(2,\C)$-connection
$$(\tilde E'',\tilde \nabla''):=(\tilde E',\tilde \nabla')\otimes(L,\zeta)^{\otimes(-1)}$$
has exponent $2\theta_i-1$ over $\omega_i$, $i=0,1,t,\infty$.
For special parameters
$$\boldsymbol{\theta}=(\frac{1}{2},\frac{1}{2},\frac{1}{2},\frac{1}{2}+\frac{\vartheta}{2}),$$
$(\tilde E'',\tilde \nabla'')$ is a Lam\'e connection (i.e. having a single
pole at $\infty$) with exponent $\vartheta$ (at $\omega_\infty$); 
as we shall see, all irreducible Lam\'e connections
can be obtained by this way. 
When $\vartheta$ is an odd integer, we note that
$\omega_\infty$ is a parabolic singular point of $\tilde\nabla''$
if, and only if, $\infty$ is parabolic for $\nabla$; when $\vartheta$ 
is even, $\omega_\infty$ is always apparent for the Lam\'e connection $\tilde \nabla''$ (never parabolic).

\section{Computing the vector bundle of an elliptic pull-back}

Under notations of section \ref{sec:MainConstruction},
we would like to determine the vector bundle $\tilde E''$ over
the elliptic curve $X$ in terms of the initial connection $(E,\nabla)$.
In fact, the construction of $\tilde E''$ only depend on $E$ and
the parabolic structure $\boldsymbol{l}=(l_0,l_1,l_t,l_\infty)$.
We restrict our attention to irreducible connections for simplicity,
although the general case could be handled by similar arguments.
The goal of this section is to prove the

\begin{thm}\label{thm:ComputeBundle}
Let $(E,\nabla,\boldsymbol{l})$ be an irreducible parabolic
connection as before and $(\tilde E'',\tilde \nabla'')$ its elliptic pull-back. 
Then $\tilde E''$ is semistable and we are in one of the following cases:
\begin{itemize}
\item $E$ is the trivial bundle and not three lines $l_i$ coincide.\newline
In particular, the cross-ratio 
$$c=\frac{l_t-l_0}{l_1-l_0}\frac{l_1-l_\infty}{l_t-l_\infty}\in\P^1$$
is well defined and we have:
$$\lambda(\tilde E'')=t\frac{c-1}{c-t}.$$
Precisely, $\tilde E''$ is undecomposable if, and only if
\begin{itemize} 
\item either $c=t$ (the diagonal case) 
\item or $c=0,1,\infty$ and only two lines $l_i$ coincide (the other two ones being mutually distinct).
\end{itemize} 
\item $E=\mathcal O(-1)\oplus\mathcal O(1)$ and none of the 
lines $l_i$ coincides with the sub-line-bundle 
$\mathcal O(1)$. Then $\tilde E''$ is the trivial bundle or the undecomposable one $E_0$ depending on the fact that all
$l_i$ lie on a sub-line-bundle $\mathcal O(-1)\hookrightarrow E$
or not.
\end{itemize} 
\end{thm}

For the application we have in mind, one should emphasize
that, along the isomonodromic deformation $(E_t,\nabla_t)$
of such a connection $(E,\nabla)$, the set of parameters $t$ 
for which $E_t$ is not trivial is always a strict (possibly empty) 
analytic subset of the parameter space $T$ (see \cite{Heu}).

Let us start the proof of Theorem \ref{thm:ComputeBundle} 
by justifying restrictions on $E$ and $\boldsymbol{l}$.
This will follow from

\begin{lem}\label{lem:Brunella}
Let $X$ be a curve of genus $g$, $E$ a vector bundle,
$\nabla:E\to E\otimes\Omega(D)$ a logarithmic connection
with reduced effective divisor $D$,  and $L\subset E$
a line bundle which is not $\nabla$-invariant. Then, the integer
$$\nu:=\deg(E)+\deg(D)+2g-2-2\deg(L)\ge0$$
is bounding the number of poles of $\nabla$ where $L$ coincides with an eigenline.
\end{lem}

Here, we mean the eigenline of the residue of $\nabla$.

\begin{proof}The composition 
$$\xymatrix{
    L \ar@{^{(}->}[rr]^{\text{inclusion}} && E \ar[r]^-{\nabla}& E\otimes\Omega(D)
    \ar@{->>}[r]^-{\text{quotient}} &E/L\otimes\Omega(D)=:L'
     }$$
defines an homomorphism of line bundles $L\to L'$,
and thus a section 
$$\phi\in\mathrm{H}^0(X,L'\otimes L^{-1})=
\mathrm{H}^0(X,\det(E)\otimes\Omega(D)\otimes L^{-2}).$$
When $L$ is not $\nabla$-invariant, then $\phi$ is a non trivial section
vanishing precisely at
\begin{itemize}
\item non singular points of $\nabla$ where $L$ is stabilized: $\nabla(L)\subset L\otimes\Omega(D)$,
\item poles of $\nabla$ where $L$ coincides with an eigenline of the residue of $\nabla$.
\end{itemize}
This gives the result.
\end{proof}

In our situation, $g=0$ and $\deg(D)=4$; we deduce that $2-2\deg(L)\ge0$ for any line bundle $L\subset E$ (since $\nabla$ is irreducible, no line bundle $L$ can be $\nabla$-invariant).
Therefore, $E$ is either trivial, or 
$\mathcal O(-1)\oplus\mathcal O(1)$. In the former caser, 
any embedding $\mathcal O\hookrightarrow E$ passes 
through at most two eigenlines; in the later case, 
$L=\mathcal O(1)$ passes through no eigenlines.

\subsection{Projective bundles}
In order to compute the modular invariant $\lambda(\tilde E'')$
for the elliptic pull-back, it is more convenient to work with the
associated projective bundle. If $E$ is a 
vector bundle over a curve $X$, then we denote by $\P(E)$
the associated $\P^1$-bundle. Another vector bundle $E'$
gives rise to the same $\mathbb  P^1$-bundle if, and only if, 
$E'=L\otimes E$ for a line bundle $L$. When $E$ and $E'$ are
both determinant free, then $L$ is a $2$-torsion point of the jacobian:
there are at most $2^{2g}$ determinant free vector bundles $E'$ giving
rise to the same $\mathbb  P^1$-bundle, where $g$ is the genus of $X$.
Line bundles $L\subset E$ are in one-to-one correspondence
with sections $\sigma:X\to\P(E)$.
The total space $S$ of $\P(E)$ fits naturally into a ruled surface
and any section $\sigma$ defines a curve on $S$ that we still denote by $\sigma$ for simplicity;
the normal bundle of $\sigma$ in $S$ identifies with the line bundle
$\det(E)\otimes L^{-2}$ where $L\subset E$ is the corresponding line bundle, so that the self-intersection is given by
$$\sigma\cdot \sigma=\deg(E)-2\deg(L).$$
Recall that $E$ is semistable if, and only if, $\deg(E)-2\deg(L)\ge 0$
for any line bundle $L\subset E$.
If $L'$ is any line bundle distinct from $L$, then the composition
$$L'\to E\to E/L\simeq \det(E)\otimes L^{-1}$$
is a non trivial homomorphism of line bundles; this defines
a non trivial holomorphic section of $\det(E)\otimes L^{-1}\otimes L'^{-1}$
vanishing at those points where $L$ and $L'$ coincide. 
For the corresponding sections $\sigma$ and $\sigma'$, we deduce
the intersection number
$$\sigma\cdot \sigma'=\deg(E)-\deg(L)-\deg(L')=\frac{1}{2}(\sigma\cdot \sigma+\sigma'\cdot \sigma')\ge 0.$$
A vector bundle $E$ is decomposable, i.e. of the form
$E=L\oplus L'$, if and only if $\P(E)$ admits
two disjoint sections $\sigma$ and $\sigma'$ (which obviously
correspond to $L$ and $L'$ respectively), or equivalently
two sections $\sigma$ and $\sigma'$ having opposite self-intersection numbers.
In this case, the $\mathbb  P^1$-bundle $\P(E)$ may be viewed
as the fibre-compactification $\overline{L'\otimes L^{-1}}$ of the line bundle $L'\otimes L^{-1}$ obtained by adding a section at infinity. Precisely, 
$\sigma$ (resp. $\sigma'$) stands for  
the null section (resp. the section at infinity) in 
$\P(E)\simeq\overline{L'\otimes L^{-1}}$. 

When $X$ is an elliptic curve, say $X:\{y^2=x(x-1)(x-t)\}$, 
and $\det(E)=\mathcal O_X$, 
then $E$ is semistable if, and only if, $\P(E)$ has a section 
having zero self-intersection. The four undecomposable bundles
defined by (\ref{E:UndecomposableBundles}) correspond to the 
same $\mathbb  P^1$-bundle $\P(E_0)$ defined by the 
unique non trivial extension 
$$0\to \mathcal O_X\to E_0\to \mathcal O_X\to 0.$$
The corresponding ruled surface $S_0$ is characterized 
by the fact that it admits one, and exactly one section
having zero self-intersection.
Following Atiyah \cite{Atiyah,Tu}, all other semistable and determinant free vector bundles are decomposable: the corresponding 
$\mathbb  P^1$-bundle takes the form 
$$\P(E)=\overline{\mathcal O_X([\omega]-[-\omega])}$$
where $\pm\omega\in X$ are the two points of the fiber
$\pi^{-1}(\lambda(E))$ (see definition (\ref{E:ModuliSemistableBundle})). We note that the modular 
invariant $\lambda(E)$ is determined by $\P(E)$ 
up to the action of the $2$-torsion points of the elliptic curve $X$:
determinant free vector bundles with $\mathbb  P^1$-bundle 
$\P(E)$ are the four semistable bundles with modular invariants
$$\lambda,\ \ \  \frac{t}{\lambda},\ \ \ \frac{\lambda-t}{\lambda-1}\ \ \ 
\text{and}\ \ \ t\frac{\lambda-1}{\lambda-t}.$$

\subsection{Ruled surfaces and elliptic pull-back}\label{subsec:RuledEllipticPullBack}
Let us now describe the construction of section 
\ref{sec:MainConstruction} in terms of ruled surfaces.
We start from a rational ruled surface 
$p:S=\P(E)\to\P^1$
equipped with the parabolic structure $\boldsymbol{l}$
defined by a point $l_i$ on the fibre $S\vert_i=p^{-1}(i)$
for $i=0,1,t,\infty$.

\noindent{\bf step 1 :} The elliptic ruled surface 
$\tilde p:\tilde S=\P(\tilde E)\to\P^1$ is obtained 
after a two-fold ramified cover 
$\Pi:\tilde S\to S$ ramifying over
the four fibres $S\vert_i$ and makes commutative the following 
diagram 
$$
\xymatrix{
    \tilde S \ar[r]^{\Pi} \ar[d]_{\tilde p}  & S \ar[d]_p \\
    X\ar[r]^{\pi} & \P^1
}
$$
We equipp $ \tilde S$ with the parabolic structure
$\tilde{\boldsymbol{l}}$ defined by $\tilde l_i=\Pi^{-1}(l_i)\in \tilde S\vert_{\omega_i}=\tilde p^{-1}(\omega_i)$ for $i=0,1,t,\infty$.

\noindent{\bf step 2 (and 3) :} The birational transformation 
$$
\xymatrix{ \tilde S \ar@{-->}[rr]^-\phi \ar[rd]_{\tilde p} && \tilde S' \ar[ld]^{\tilde p} \ar@{=}[r]^{\text{step 3}}&\tilde S''\\ & X }
$$
is obtained by blowing-up the four points $\tilde l_i$, and then blowing-down 
the strict transform of the four fibers. Step 3 is not relevant from the projective
point of view since we just multiply $\tilde E'$ by a line bundle
in order to obtain $\tilde E''$. Therefore, $\tilde S''=\tilde S'$.

\subsection{Elliptic ruled surfaces and elementary transformations}
From Lemma \ref{lem:Brunella}, we know that the parabolic
ruled surface $(S,\boldsymbol{l})$ we start with is
\begin{itemize}
\item either $S=\P^1\times\P^1$ and 
not three points $l_i\in S$ lie on the same horizontal line,
\item or $S=\mathbb F_2$ is the second Hirzebruch surface, the total space of $\P(\mathcal O(-1)\oplus\mathcal O(1))$, and no point $l_i$ lie on the ``negative'' section $\sigma$ corresponding to $\mathcal O(1)$.
\end{itemize}
In each case, we denote by $\tilde S'$ the corresponding
elliptic pull-back constructed in section 
\ref{subsec:RuledEllipticPullBack}.

\begin{prop}\label{P:ElmTransfRuled}  When $S=\P^1\times\P^1$, choose a coordinate $w$ and denote by 
$$c=\frac{w_t-w_0}{w_1-w_0}\frac{w_1-w_\infty}{w_t-w_\infty}\in\P^1$$
the cross-ratio where $w_i$ are defined by $l_i=(i,w_i)$. Then we have:
\begin{itemize}
\item when $c\not=0,1,t,\infty$, then $\tilde S'$ is the decomposable
ruled surface
$$\tilde S'=\overline{\mathcal O_X([\omega]-[-\omega])}$$
where $\pm \omega=(x,\pm y)\in X$ are the two
points over $x=t\frac{c-1}{c-t}$;
\item when $c=t$ (all four points $l_i$ lie on the same irreducible bidegree $(1,1)$-curve) then $\tilde S'$ is the undecomposable ruled surface $S_0$;
\item when $c=0,1,\infty$, then at least two points $l_i$ lie on 
the same horizontal line; if the two other points lie on 
another horizontal line, then $\tilde S'$ is the trivial bundle;
else, $\tilde S'$ is the  undecomposable ruled surface $S_0$.
\end{itemize}
When $S=\mathbb F_2$, then we have:
\begin{itemize}
\item when the four points $l_i$ lie on a section having $+2$ self-intersection (i.e. induced by any embedding $\mathcal O_{\P^1}(-1)\hookrightarrow\mathcal O_{\P^1}(-1)\oplus\mathcal O_{\P^1}(1)$)
then $\tilde S'$ is the trivial bundle;
\item else, $\tilde S'$ is the  undecomposable ruled surface $S_0$.
\end{itemize}
\end{prop}

\begin{proof}We first consider the generic case $S=\P^1\times\P^1$ and $c\not=0,1,t,\infty$. One can choose the vertical
coordinate $w$ such that 
$$l_0=(0,0),\ \ \ l_1=(1,1),\ \ \ 
l_t=(t,c)\ \ \ \text{and}\ \ \ l_\infty=(\infty,\infty).$$
One easily check by computation that there is a unique  curve $C\subset\P^1\times\P^1$ of bidegree $(2,2)$ intersecting each fibre $x=i$ at the point $l_i$ with multiplicity $2$. 
The equation 
for $C$ is given by $F=0$ where
\begin{equation}\label{Eq:Fcross}
F(x,w)=((c-t)x-t(c-1))w^2+2(t-1)cxw-cx((c-1)x-(c-t)).
\end{equation}
The discriminant with respect to the $w$-variable, given by 
$$\mathrm{disc}_y(F)=4c(c-1)(c-t)x(x-1)(x-t),$$
is not identically vanishing, with simple roots; 
we deduce that the curve $C$ is reduced, irreducible and smooth. 
Its lifting $\tilde C$ to the trivial bundle $X\times\P^1$ splits into the union 
of $2$ distinct sections $\tilde\sigma_0$ and $\tilde\sigma_\infty$ intersecting 
exactly at the $4$ points $\tilde l_i$ without multiplicity. 
After elementary 
transformations with center $\tilde l_i$, we obtain two disjoint sections
$\sigma_0$ and $\sigma_\infty$ of $\tilde S'$; we already deduce that $\tilde S'$
is the compactification of a line bundle. In order to determine this line bundle,
consider the horizontal section $w=\infty$ of $\P^1\times\P^1$
passing through $l_\infty$: it intersects the $(2,2)$-curve $C$ in $2$ points, namely 
$l_\infty$ and the point $s=(x,\infty)$ with coordinate $x=t\frac{c-1}{c-t}$;
it lifts to a section $\tilde\sigma$ of $X\times\P^1$ intersecting
$\tilde\sigma_0$ over $0$ and, say,  $\omega$, and intersecting $\tilde\sigma_\infty$ over $0$ 
and $-\omega$ where $\pm \omega\in X$ are the two points of $X$ over $x$. 
After elementary transformations with centers $\tilde l_i$,
we derive a section $\sigma$ of $\tilde S'$ intersecting $\sigma_0$ at $\omega$
and $\sigma_\infty$ at $-\omega$: we thus obtain $\tilde S'=\overline{\mathcal O([\omega]-[-\omega])}$. This is resumed in  
picture \ref{figure:1}.

\begin{figure}[htbp]
\begin{center}

\input{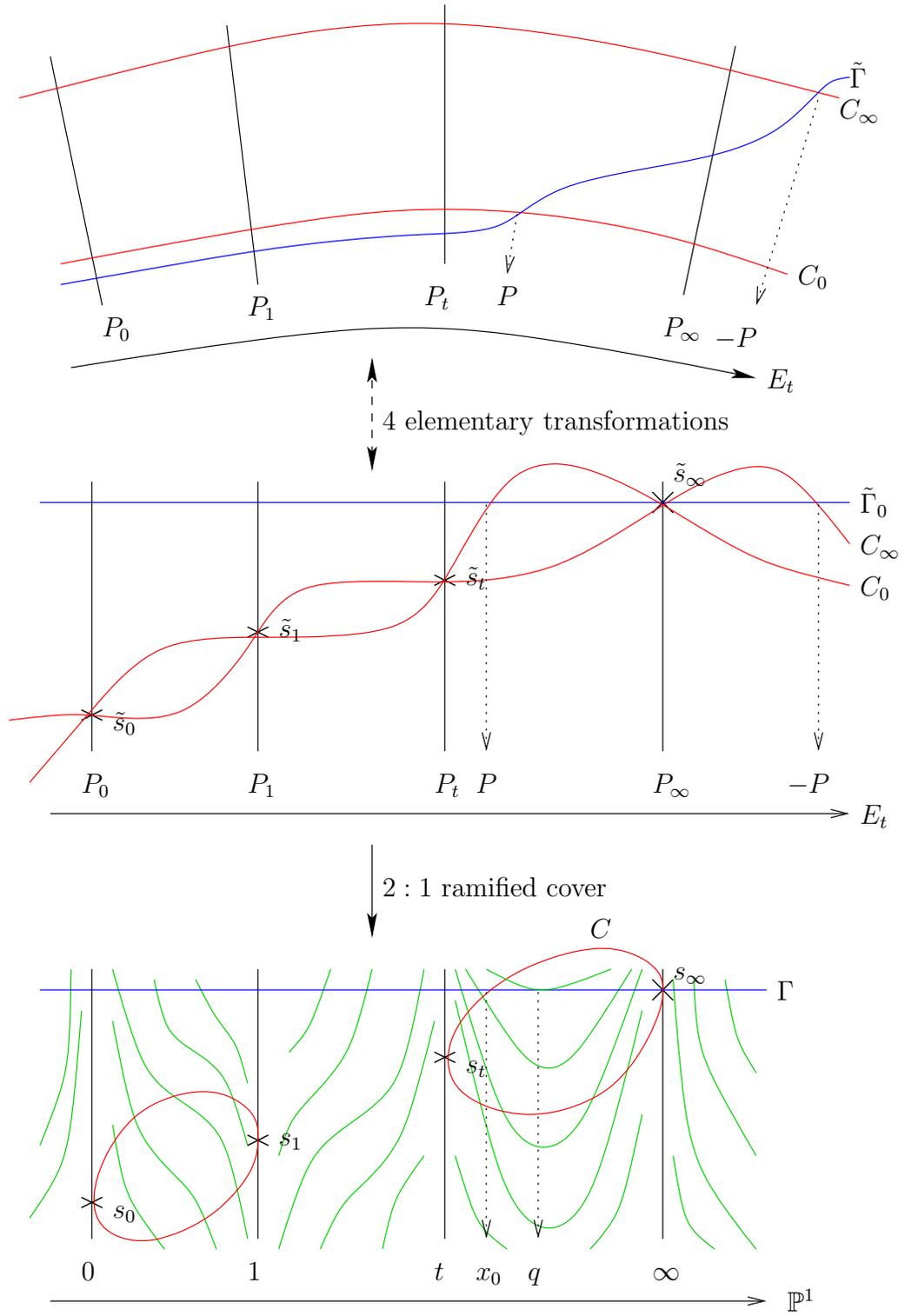}
 
\caption{Lifting $\P(\nabla)$ on the elliptic curve $E_t$}
\label{figure:1}
\end{center}
\end{figure}

When $c=t$, the curve $F=0$ degenerates to twice the $(1,1)$-curve
passing through all $l_i$, namely the diagonal section $\sigma(x)=x$. Its lifting
to $\tilde S$ is the graph $\tilde\sigma$ of the two-fold cover $X\to\P^1$ having self-intersection $+4$. After elementary 
transformations, we deduce a section $\tilde\sigma'$ of $\tilde S'$
having zero self-intersection. More precisely, the normal bundle of 
$\tilde\sigma'$ is the trivial bundle. Indeed, the section 
$\tilde\sigma$ is induced in $\tilde S=\P(\tilde E)$, $\tilde E=\mathcal O_X\oplus\mathcal O_X$, by the line bundle $\tilde L\subset\tilde E$ generated by its meromorphic
section $(x,y)\mapsto\begin{pmatrix}x\\1\end{pmatrix}$, whose
divisor is $-2[\omega_\infty]$; the normal bundle of $\tilde\sigma$
in $\tilde S$ is therefore given by 
$$\mathcal N_{\tilde\sigma}=\det(\tilde E)\otimes\tilde L^{-2}=\mathcal O_X\otimes\mathcal O_X(-2[\omega_\infty])^{-2}=\mathcal O_X(4[\omega_\infty]).$$
After elementary transformations, we obtain
$$\mathcal N_{\tilde\sigma'}=\det(\tilde E')\otimes L'^{-2}$$
$$=\mathcal O_X([\omega_0]+[\omega_1]+[\omega_t]+[\omega_\infty])\otimes\mathcal O_X([\omega_0]+[\omega_1]+[\omega_t]-[\omega_\infty])^{-2}$$
$$=\mathcal O_X.$$
If $\tilde S'$ were decomposable, it would be the trivial $\P^1$-bundle,
what we have now to exclude. Consider the centers 
$\tilde l_i'\in\tilde S'$ of the inverse elementary transformations.
If $\tilde S'$ were the trivial $\P^1$-bundle, most 
of horizontal sections would avoid the four points $\tilde l_i'$
and would define, back to $\tilde S=X\times \P^1$, 
a section having $-4$ self-intersection, impossible.
Thus $\tilde S'$ is the undecomposable bundle $S_0$.

We now assume $c=\infty$; the other cases $c=0$ or $1$
are similar. We are in one of the following three cases:
\begin{itemize}
\item $w_0=0$, $w_1=1$ and $w_t=w_\infty=\infty$;
\item $w_0=w_1=0$, $w_t=1$ and $w_\infty=\infty$;
\item $w_0=w_1=0$ and $w_t=w_\infty=\infty$.
\end{itemize}
The last case is easy since the three horizontal sections 
$\sigma_0$, $\sigma_1$ and $\sigma_\infty$, respectively 
defined by $w=0$, $w=1$ and $w=\infty$, are transformed
in the elliptic pull-back $\tilde S'$ into disjoint
sections $\tilde\sigma_0'$ and $\tilde\sigma_\infty'$,
and a third section $\tilde\sigma_1'$ that intersects
$\tilde\sigma_0'$ at $\omega_t$ and $\omega_\infty$
and $\tilde\sigma_\infty'$ at $\omega_0$ and $\omega_1$.
We promptly deduce that 
$$\tilde S'=\overline{\mathcal O_X([\omega_t]+[\omega_\infty]-[\omega_0]-[\omega_1])}=X\times\P^1.$$
We now study the first case, where only $w_t$ and $w_\infty$ coincide; it is similar to the diagonal case. The section $w=\infty$ 
of $S$ defines a section $\tilde \sigma_\infty'$ of $\tilde S'$ having 
even trivial normal bundle. If $\tilde S'$ were decomposable, it would be the trival bundle, and a generic constant section $\tilde\sigma'$
would provide a section $\tilde\sigma$ of the trivial bundle $\tilde S$
having $-4$ self-intersection, contradiction.
Thus $\tilde S'$ is the undecomposable bundle $S_0$.

Finally, let us consider the case where $S=\mathbb F_2$.
Like before, one easily shows that the exceptional section
$\sigma_\infty$ induced by $\mathcal O_{\P^1}(1)$ 
yields a section $\tilde\sigma_\infty'$ of $\tilde S'$ having
trivial normal bundle. Again, like in the diagonal case,
consider the centers $l_i'\in\tilde S'$ of elementary 
transformations inversing $\phi$: they are contained 
in $\tilde\sigma_\infty'$. If $\tilde S'$ is the trivial bundle,
then constant sections $\tilde\sigma'$ give rise to a pencil of sections
$\tilde\sigma$ of $\tilde S$ having the four points $\tilde l_i$
as base points. A special member of the pencil is given by
the union of $\tilde\sigma_\infty$ and the four fibres over $\omega_i$.
In fact, this pencil consists in all sections
of $\tilde S$ passing through all $\tilde l_i$ that do not intersect
$\tilde\sigma_\infty$ (plus the special one). Now, the elliptic involution 
$\tau:(x,y)\mapsto(x,-y)$ permutes those sections, therefore acting
on the parameter space $\P^1$: there are at least $2$
fixed points, namely $\tilde\sigma_\infty$ and another section 
$\tilde\sigma_0$ that can be pushed down as a section $\sigma_0$
of $S$; by construction, $\sigma_0$ passes through all points $l_i$
and does not intersect $\sigma_\infty$: it's a $+2$-curve as required.
Conversely, when all $l_i$ are contained in a $+2$-curve $\sigma_0$,
we deduce a second section $\tilde\sigma_0'$ of $\tilde S'$
yielding a trivialization. We note that the pencil considered
above comes from a pencil not of sections of $S$, but of curves
intersecting twice a fibre.
\end{proof}

\section{Isomonodromic deformations and the Painlev\'e VI equation}

In this section, we recall how isomonodromic deformations 
of logarithmic $\mathrm{sl}(2,\C)$-connections 
$(E_t,\nabla_t)$
over the $4$-punctured sphere are parametrized by Painlev\'e VI
solutions, and how we can use this parame\-trization to compute
the variation of the bundle $E_t''$ of the corresponding elliptic 
pull-back. Let us first recall what an isomonodromic deformation is.

\subsection{Isomonodromic deformations and flat connections}
Let $X_0$ be a complex projective curve and $(E_0,\nabla_0)$ 
be the data of a rank $2$ vector bundle over $X_0$ equipped
with a (flat) logarithmic $\mathrm{sl}(2,\C)$-connection 
having (reduced) effective polar divisor $D_0$. 
Given a topologically trivial analytic deformation $(X_t,D_t)$
of the punctured curve,
there is a unique deformation $(E_t,\nabla_t)$ of both the vector bundle and the connection having constant monodromy data. 
This just follows from the Riemann-Hilbert correspondence
(proposition \ref{prop:RHparabolic} with parameter).
Monodromy data consists in the monodromy representation,
completed by the ``parabolic structure'' at apparent singular points.
Equivalently, if we denote by $\boldsymbol X$
the total space of the deformation, $\boldsymbol D\subset \boldsymbol X$ the (smooth) divisor, then
the deformation $(E_t,\nabla_t)$ is induced by the unique 
flat logarithmic $\mathrm{sl}(2,\C)$-connection $(E,\nabla)$
with polar divisor $\boldsymbol D$ inducing $(E_0,\nabla_0)$ on the slice $X_0$. 
In this paper, we will consider the case of the $4$-punctured sphere
and the once-punctured torus whose deformation are parametrized 
by the corresponding Teichm\"uller spaces that are
both isomorphic to the Poincar\'e half-plane $\mathbb H$. 
Precisely, we start with the isomonodromic deformation 
of a logarithmic $\mathrm{sl}(2,\C)$-connection 
$(E_t,\nabla_t)$ over the Riemann sphere with poles at 
$0$, $1$, $t$ and $\infty$ where $t$ varrying in the universal
cover $T\simeq\mathbb H\to \P^1\setminus\{0,1,\infty\}$
and we consider the deformation $(\tilde E_t,\tilde\nabla_t)$
of the corresponding elliptic pull-back, as constructed in section \ref{sec:MainConstruction}. As one can easily check, 
$(\tilde E_t,\tilde\nabla_t)$ is still an isomonodromic deformation 
of a logarithmic connection over the Legendre family of elliptic curves 
$X_t$ with poles contained into the ramification
locus $\{\omega_0,\omega_1,\omega_t,\omega_\infty\}$ of the elliptic curve;
the parameter space $T$ is now understood as the Teichm\"uller
space of the torus. For special parameters 
$\boldsymbol{\theta}=(\frac{1}{2},\frac{1}{2},\frac{1}{2},\frac{1}{2}+\frac{\vartheta}{2})$, we obtain the isomonodromic deformation of a 
Lam\'e connection. 

\subsection{Painlev\'e VI equation and fuchsian equations}\label{ss:PainleveFuchs}
Although we do not really need it, it is interesting to recall
how the Painlev\'e VI equation was originaly derived as isomonodromic equation for fuchsian
projective structures on the $4$-punctured sphere with 
one extra branch point. After normalizing the singular points 
as $0$, $1$, $t$ and $\infty$ by a Moebius transformation,
the corresponding
fuchsian $2$nd order differential equation $u_{xx}+f(x)u_x+g(x)u=0$ 
takes the form
\begin{equation}\label{E:Scalar4+1}
\left\{\begin{matrix}
f(x)&=&{\frac {1-\kappa_0}{x}}+{\frac {1-\kappa_1}{x-1}}
+{\frac {1-\kappa_t}{x-t}}-{\frac {1}{x-q}}\\
g(x)&=& \frac{-{\frac {t(t-1) H}{x-t}}+{\frac {q(q-1)p}{x-q}}
+\rho(\kappa_\infty+\rho)}{x(x-1)}
\end{matrix}\right.
\end{equation}
Here, $q\not\in\{0,1,t,\infty\}$ is the branch point, $\kappa_i$
is the local exponent at $i=0,1,t,\infty$ and $\rho$ is fixed by the relation
\begin{equation}\label{E:Rho}
\kappa_0+\kappa_1+\kappa_t+\kappa_\infty+2\rho=1.
\end{equation}
Note that parameters $p$ and $H$ are residues of $g$
\begin{equation}
H=-\mathrm{Res}_{x=t}g(x)\ \ \ \text{and}\ \ \ p=\mathrm{Res}_{x=q}g(x).
\end{equation}
The singular point $q$ has exponent $2$; it is apparent, 
i.e. a branch point of the projective chart, if and only if
the parameter $H$ is given by
\begin{equation}\label{E:Hamiltonien}
\begin{matrix}H=\frac{q(q-1)(q-t)}{ t(t-1)}\left(p^2-({\kappa_0\over q}+{\kappa_1\over q-1}
+{\kappa_t-1\over q-t})p
+{\rho(\kappa_\infty+\rho)\over q(q-1)}\right)\end{matrix}
\end{equation}
Under these assumptions, the local charts $\phi=\frac{u_1}{u_2}$, where $u_1$ and $u_2$ run over independant solutions 
of (\ref{E:Scalar4+1}), defines a projective atlas on the complement
of $0$, $1$, $t$, $q$ and $\infty$ in the Riemann sphere. 
At each singular point $i=0,1,t,\infty$, one of the projective 
charts takes the form
$\phi=z^{\kappa_i}$ (or possibly $\phi=\frac{1}{z^m}+\log(z)$
when $\kappa_i=\pm m\in\Z_{>0}$) for a convenient local coordinate $z$ at $i$; at $q$, one of the projective charts takes
the form $\phi=z^2$, a simple branch point.
Conversely, any projective structure on the Riemann sphere having
five singularities with moderate growth, one of which being a simple
branch point, is conjugated by a Moebius transformation to an element of the above family. Such a projective structure is characterized
by the following data
\begin{itemize}
\item the position $t$ and $q$ of the singular points,
\item the exponents $\boldsymbol{\kappa}=(\kappa_0,\kappa_1,\kappa_t,\kappa_\infty)$
\item and the monodromy representation of a projective chart $\phi$,
up to conjugacy.
\end{itemize}
A small deformation of equation (\ref{E:Scalar4+1}) with moving
singular points $t$ and $q$ is said isomonodromic when the 
projective charts have constant monodromy representation, 
up to conjugacy. They are characterized by the classical

\begin{thm}[Fuchs-Malmquist]\label{T:FuchsMalmquist}
A deformation of equation (\ref{E:Scalar4+1}) parametrized
by the position $t$ of the singular point is isomonodromic if, 
and only if, exponents $\kappa_i$ are fixed and parameters
$(p(t),q(t))$ satisfy the non autonomous hamiltonian system
\begin{equation}\label{eq:HamSyst}
\frac{\partial q}{\partial t}=\frac{\partial H}{\partial p}\ \ \ \text{and}\ \ \ 
\frac{\partial p}{\partial t}=-\frac{\partial H}{\partial q}.
\end{equation}
\end{thm}

The first hamiltonian equation (\ref{eq:HamSyst}) rewrites
\begin{equation}\label{eq:pFromH}
p=\frac{1}{2}\left(\frac{t(t-1)}{q(q-1)(q-t)}\frac{\partial q}{\partial t}
+\frac{\theta_0}{q}+\frac{\theta_1}{q-1}+\frac{\theta_t-1}{q-t}\right)
\end{equation}
Now, substituting (\ref{eq:pFromH}) in the second equation
(\ref{eq:HamSyst}) yields the Painlev\'e VI equation
(\ref{E:PainleveEquation}) with parameter $\boldsymbol{\kappa}$.
From the chronological point of view, the Painlev\'e VI
equation was first derived by Fuchs; the hamiltonian form 
was discovered later by Malmquist.

\subsection{Painlev\'e VI equation and fuchsian systems}\label{ss:PainleveVISystems}
Let us now recall how Painlev\'e VI solutions correspond
to isomonodromic deformations of logarithmic 
$\mathrm{sl}(2,\C)$-connections $(E_t,\nabla_t)$ 
with singular points $0$, $1$, $t$ and $\infty$ over the 
Riemann sphere. Let $(E_t,\nabla_t)$ be such a deformation
and assume moreover that it is irreducible (this does not depend
on $t$, only on the monodromy). By Lemma
\ref{lem:Brunella}, the underlying bundle $E_t$ is either trivial, 
or $\mathcal O(-1)\oplus\mathcal O(1)$, but it turns out
that $E_t$ will be the trivial bundle for all but a discrete
subset of the parameter $T$:
there are no non trivial irreducible isomonodromic deformations 
of such connections on the bundle 
$\mathcal O(-1)\oplus\mathcal O(1)$ (see \cite{Heu}).
It is thus enough to consider isomonodromic
deformations of $\text{sl}(2,\C)$-Fuchsian systems
\begin{equation}\label{E:FuchsianSystem4}
{d Y\over dx}=\left({A_0\over x}+{A_1\over x-1}+{A_t\over x-t}\right)Y,\ \ \ 
A_i\in\mathrm{sl}(2,\C).
\end{equation}
The residual matrix of the singular point at $x=\infty$ is given by
\begin{equation}\label{E:ResiduInfini}
A_0+A_1+A_t+A_\infty=0.
\end{equation}
Let $\pm {\theta_i\over 2}$ denote the eigenvalues of $A_i$:
\begin{equation}\label{E:ValeursPropres}
A_i=\left(\begin{matrix}a_i&b_i\\c_i&-a_i\end{matrix}\right)\ \ \ \text{with}\ \ \ 
a_i^2+b_ic_i={\theta_i^2\over 4},\ \ \ i=0,1,t,\infty
\end{equation}
After change of variable, $Y:=MY$ with $M\in\text{SL}(2,\C)$, 
we normalize
\begin{equation}\label{E:NormalizationAinfinity}
A_\infty=
\left(\begin{matrix}{\theta_\infty\over 2}&0\\ \ast&-{\theta_\infty\over
2}\end{matrix}\right).
\end{equation}
We exclude the case where $A_\infty=0$: $\nabla$ is assumed singular at $\infty$. Then, we have

\begin{thm}\label{T:FuchsSystem}
A small deformation $A_i=A_i(t)$ of the system (\ref{E:FuchsianSystem4})
normalized by (\ref{E:NormalizationAinfinity}) is isomonodromic if, 
and only if,
the eigenvalues $\pm{\theta_i\over 2}$ 
are constant and $q:={tb_0 \over tb_0 +(t-1)b_1}$ satisfies 
\begin{equation}\label{E:qEquation}
{dq\over dt}=-2a_0{q-1\over t-1}-2a_1{q\over t}+(1-\theta_\infty){q(q-1)\over t(t-1)}
\end{equation}
and the Painlev\'e VI equation (\ref{E:PainleveEquation})
with parameter 
$$\boldsymbol\kappa=(\kappa_0,\kappa_1,\kappa_t,\kappa_\infty)
=(\theta_0,\theta_1,\theta_t,\theta_\infty-1).$$
\end{thm}

Let us see first how to deduce Theorem (\ref{T:FuchsSystem})
from Fuchs-Malmquist Theorem (\ref{T:FuchsMalmquist}).

\begin{proof}The vector $\begin{pmatrix}0\\ 1\end{pmatrix}$ is an eigenvector
for the eigenvalue $-\frac{\theta_\infty}{2}$ at $\infty$; by irreducibility
of system (\ref{E:FuchsianSystem4}), it is not invariant and
can be choosen as a cyclic vector to derive a scalar fuchsian
equation. Namely, if $Y=\begin{pmatrix}y_1\\ y_2\end{pmatrix}$
if a solution of (\ref{E:FuchsianSystem4}), then 
$$u:=\sqrt{x^{\theta_0}(x-1)^{\theta_1}(x-t)^{\theta_t}}y_1$$
satisfies the scalar equation $(\ref{E:Scalar4+1})$ with
exponents 
\begin{equation}\label{E:KappaTheta}
(\kappa_0,\kappa_1,\kappa_t,\kappa_\infty)=(\theta_0,\theta_1,\theta_t,\theta_\infty-1)
\end{equation}
and parameters
\begin{equation}\label{E:qDefinition}
q={tb_0 \over tb_0 +(t-1)b_1}.
\end{equation}
and
\begin{equation}\label{eq:pFromAi}
p=\frac{a_0+\frac{\theta_0}{2}}{q}+\frac{a_1+\frac{\theta_1}{2}}{q-1}+\frac{a_t+\frac{\theta_t}{2}}{q-t}.
\end{equation}
In fact, the Darboux coordinates have the following 
interpretation: the point $x=q$ is the unique other point at
which $\begin{pmatrix}0\\ 1\end{pmatrix}$ is again an eigenvector
of system (\ref{E:FuchsianSystem4}) and 
$-p+\frac{\theta_0}{2q}+\frac{\theta_1}{2(q-1)}+\frac{\theta_t}{2(q-t)}$
is the corresponding eigenvalue. From Theorem (\ref{T:FuchsMalmquist}), we deduce that a deformation of 
system (\ref{E:FuchsianSystem4}) is isomonodromic 
if, and only if, the auxiliary variables $p$ and $q$ defined
by (\ref{E:qDefinition}) and (\ref{eq:pFromAi}) satisfy
the hamiltonian system (\ref{eq:HamSyst}) where $H(t,p,q)$ 
is the non autonomous hamiltonian defined by 
(\ref{E:Hamiltonien}).
\end{proof}

Let us now explain how to uniquely reconstruct 
(\ref{E:FuchsianSystem4})
up to gauge transformation from a solution $q(t)$ of Painlev\'e VI
equation. First introduce auxiliary variable $p$ by 
formula (\ref{eq:pFromH}). This gives us a unique scalar
equation (\ref{E:Scalar4+1}) from which one can reconstruct
a fuchsian system by standart method. The resulting system
is defined (up to gauge transformation) by (\ref{E:FuchsianSystem4})
with equations (\ref{E:ValeursPropres}) and
\begin{equation}\label{E:pqcoeffAi}
\left\{\begin{matrix}
a_0&=&\frac{\boldsymbol p}{t}-\frac{\kappa_0}{2}\\
a_1&=&-\frac{\boldsymbol p}{t-1}-\frac{q-1}{t-1}(\rho+\kappa_\infty)-\frac{\kappa_1}{2}\\
a_t&=&\frac{\boldsymbol p}{t(t-1)}+\frac{q-t}{t-1}(\rho+\kappa_\infty)-\frac{\kappa_t}{2}
\end{matrix}\right.
\ \ \ \text{and}\ \ \ 
\left\{\begin{matrix}
b_0&=&-\frac{q}{t}\\
b_1&=&\frac{q-1}{t-1}\\
b_t&=&-\frac{q-t}{t(t-1)}
\end{matrix}\right.
\end{equation}
where $\boldsymbol p=q(q-1)(q-t)p$ and $\rho$ is defined by (\ref{E:Rho}).
The coefficients $c_i$ of the system are immediately deduced
from equations (\ref{E:ValeursPropres}).
The standart formulae given by Jimbo and Miwa 
(see \cite{Boalch}, pages 199-200) assume $\theta_\infty\not=0$
so that the matrix $A_\infty$ can be further normalized as a diagonal matrix by additional gauge transformation; the resulting formulae look much more complicated than above ones. The way we obtain 
formulae (\ref{E:pqcoeffAi}) is described in sections
\ref{sec:RiccatiF2}, \ref{sec:RiccatiF1} and \ref{sec:RiccatiF0}.

\subsection{The vector bundle of an elliptic pull-back}
Coming back to our initial problem, in view of applying Theorem
\ref{thm:ComputeBundle}, we would like parametrize the 
parabolic structure $\boldsymbol l$ induced $\nabla_t$ 
(or equivalently system (\ref{E:FuchsianSystem4})) 
by means of the Painlev\'e VI 
transcendant $q(t)$ that parametrizes the deformation.
From formulae (\ref{E:ValeursPropres}) and (\ref{E:pqcoeffAi}),
we deduce that the eigenline $l_i$ associated to the eigenvalue
$-\frac{\theta_i}{2}$ over the pole $i=0,1,t$ is given by
$$l_i=(-b_i:a_i+\frac{\theta_i}{2})$$
which gives
\begin{equation}\label{E:Formuleli}
\left\{\begin{matrix}
l_0&=&(1:\frac{\boldsymbol p}{q})\\
l_1&=&(1:\frac{\boldsymbol p}{q-1}+\rho+\kappa_\infty)\\
l_t&=&(1:\frac{\boldsymbol p}{q-t}+(\rho+\kappa_\infty)t)\\
l_\infty&=&(0:1)
\end{matrix}\right.
\end{equation}
(recall that $l_\infty$ has been normalized by (\ref{E:NormalizationAinfinity})). These expressions for 
$l_i$ are not valid anymore for a gauge equivalent 
system (\ref{E:FuchsianSystem4}), for instance through
Jimbo-Miwa normalization, but their cross-ratio
\begin{equation}\label{E:pqcrossratio}
c=\frac{l_t-l_0}{l_1-l_0}\frac{l_1-l_\infty}{l_t-l_\infty}
=t\frac{(q-1)p+\rho+\kappa_\infty}{(q-t)p+\rho+\kappa_\infty}
\end{equation}
only depend on system (\ref{E:NormalizationAinfinity})
up to gauge transformation. We note that formula 
(\ref{E:pqcrossratio}) gives an elegant definition 
of the auxiliary variable $p$ in terms of the parabolic structure
of the connection, $q$ and $t$.

\begin{cor}\label{cor:TuInvariantGeneric}
Let $(E_t,\nabla_t)$ be the isomonodromic 
deformation defined by the Painlev\'e VI solution $q(t)$
like above and let $(\tilde E_t'',\tilde\nabla_t'')$ be the 
elliptic pull-back of the deformation. Then $\tilde E_t''$
is semistable and has invariant 
$$\lambda(\tilde E_t'')=q(t)+\frac{\rho+\kappa_\infty}{p(t)}.$$
\end{cor}

\begin{proof}By construction, $E_t$ is the trivial bundle.
Following Theorem \ref{thm:ComputeBundle}, $\tilde E_t''$
is semi-stable and has invariant
$$\lambda(\tilde E_t'')=t\frac{c-1}{c-t}=q+\frac{\rho+\kappa_\infty}{p}.$$
\end{proof}

All computations above are valid only under generic assumptions
that $q\not=0,1,t,\infty$. On the other hand, it is well known that 
constant solutions $q(t)\equiv0$, $1$, $t$ or $\infty$ correspond 
to isomonodromic deformation of reducible connections. Thus,
Corollary \ref{cor:TuInvariantGeneric} is enough to prove 
Theorem \ref{T:Main}. However, we can be more precise 
and check at special values of $q$ and $p$ if the bundle $\tilde E''$
is undecomposable or not.

\section{The geometry of Painlev\'e VI equations}\label{S:InitialConditions}

Here, we introduce some moduli space $\mathcal M^{\boldsymbol\theta}_{t_0}$ of $\mathrm{sl}(2,\C)$-connections with poles 
at $0$, $1$, $t_0$ and $\infty$ and eigenvalues $\boldsymbol\theta$.
It will contain all irreducible connections. This space originally appeared in the work \cite{Okamoto1} of Okamoto to construct
a good space of initial conditions for the Painlev\'e VI equation
from which the Painlev\'e property can be read off geometrically.
It is identified in \cite{IIS} with a moduli space of connections,
extending the dictionary established in section 
\ref{ss:PainleveVISystems}. After recalling this construction,
we will use it to determine the vector bundle $\tilde E_t$
of an elliptic pull-back for special values of $p$ and $q$.

\subsection{Okamoto's space of initial conditions}
The Painlev\'e property, characterizing the Painlev\'e equations
among differential equations $\lambda''=F(t,\lambda,\lambda')$,
says that {\it all Painlev\'e VI solutions can be analytically 
continuated as meromorphic solutions along any path avoiding 
$0$, $1$ and $\infty$}. Painlev\'e VI solutions become meromorphic
and global on the universal cover of the $3$-punctured sphere.

The (naive) space of initial conditions $\C^2\ni(q(t_0),q'(t_0))$
at some point $t_0\in\P^1\setminus\{0,1,\infty\}$ for the 
Painlev\'e VI equation fails to describe all solutions at the neighborhood of $t_0$. Painlev\'e VI solutions are meromorphic 
and some of them have a pole at $t_0$; we have to add them.
The good space of initial conditions is
\begin{equation}\label{def:SpaceOkamoto}
\mathcal M^{\boldsymbol\theta}_{t}:=\left\{\text{germs of meromorphic $P_{VI}^{\boldsymbol\theta}$-solutions at }t\right\}
\end{equation}
In order to construct it, Okamoto considers the phase portrait
of the Painlev\'e VI equation in variables 
$(t,q,q')\in\left(\P^1\setminus\{0,1,\infty\}\right)\times\C^2$ (introducing auxiliary equation
$\frac{dq}{dt}=q'$): it is defined by a rational vector field
that defines a singular holomorphic foliation on any rational compactification. For instance, we can start with 
$\P^1\times\P^2$ and observe that the
singularities of the foliation are located in special fibers
$t=0,1,\infty$ (that we don't care) and at the infinity of 
the $\P^2$-factor. The latter ones are degenerate,
located along a one dimensional section for the $t$-projection;
we have to blow-up this section in order to reduce the degeneracy 
of the singular points.
After $9$ successive blowing-up like this, Okamoto obtains a fibre bundle
\begin{equation}\label{E:PhaseSpaceBundleCompact}
t:\overline{\mathcal M^{\boldsymbol\theta}}\to
\P^1\setminus\{0,1,\infty\}
\end{equation}
(we just ignore what happens over $t=0,1,\infty$) which is not
locally trivial as an analytic bundle (it is as a topological bundle).
Some non vertical divisor $Z\subset\overline{\mathcal M^{\boldsymbol\theta}}$ consists in vertical leaves (with respect to $t$-projection)
and singular points. On the complement $\mathcal M^{\boldsymbol\theta}:=\overline{\mathcal M^{\boldsymbol\theta}}\setminus Z$
of this divisor,
the Painlev\'e foliation is transversal to $t$ inducing
a local analytic trivialization of the bundle 
\begin{equation}\label{E:PhaseSpaceBundle}
t:\mathcal M^{\boldsymbol\theta}\to
\P^1\setminus\{0,1,\infty\}
\end{equation}
By construction, the fibre $\mathcal M_{t}^{\boldsymbol\theta}$
at any point $t\not=0,1,\infty$ may be interpreted as the set
of germs of meromorphic $P_{VI}^{\boldsymbol\theta}$-solutions;
actually, for special parameters $\boldsymbol\theta$, there are leaves
staying at the infinity of the affine chart $(q,q')$ that cannot be viewed
as meromorphic solutions, but better as ``constant solutions $q\equiv\infty$''. The divisor $Z$ actually coincides with the reduced polar
divisor of the closed $2$-form 
defined in the affine chart by 
$$dt\wedge dH+dp\wedge dq$$ 
where $H$ is defined by (\ref{E:Hamiltonien}).
The kernel of this $2$-form defines the Painlev\'e foliation.

We now describe the parameter space 
$\mathcal M_{t}^{\boldsymbol\theta}$ starting from
the Hirzebruch ruled surface $\mathbb F_2$. 
Define the reduced divisor
$Z_t\subset\mathbb F_2$ as the union of the section
$\sigma:\P^1\to\mathbb F_2$ having $-2$ self-intersection
together with the $4$ fibres
over $0$, $1$, $t_0$ and $\infty$. Next, we fix $2$ points
on each vertical component of $Z_t$, none of them lying
on the horizontal one. After blowing-up
those $8$ points, we obtain the compact space 
$\overline{\mathcal M_t^{\boldsymbol\theta}}$; still denote
by $Z_t$ the strict transform of the divisor. The complement
$\mathcal M_t^{\boldsymbol\theta}:=\overline{\mathcal M_t^{\boldsymbol\theta}}\setminus Z_t$ is the space of initial conditions.
It remains to define the position of the $8$ points in function
of $\boldsymbol{\theta}$ and $t$ (see section \ref{sec:ModuliSpace}).

\subsection{Projective structures and Riccati foliations}\label{sec:ProjStrucRiccati}
We go back to the approach of R. Fuchs
where Painlev\'e transcendants parametrize 
isomonodromic deformations of fuchsian projective
structures with $4+1$ singular points (see section \ref{ss:PainleveFuchs}). Such a structure can be defined 
by the fuchsian $2$nd order differential equation
(\ref{E:Scalar4+1}). One can also define it by the data
of a logarithmic $\mathrm{sl}(2,\C)$-connection $(E,\nabla)$
together with a sub line bundle $L\subset E$ which is not $\nabla$-invariant, playing the role of a cyclic vector (see section \ref{ss:PainleveVISystems}). Such a data is called
an $\mathrm{sl}(2,\C)$-oper in \cite{BeilinsonDrinfeld}.
A more geometrical picture inspired by the works of Ehresman 
is given by the projective oper defined by the triple 
$(\P(E),\P(\nabla),\sigma)$
where $\P(E)$ is the associate $\P^1$-bundle,
$\P(\nabla)$ the induced projective connection
and $\sigma:\P^1\to \P(E)$ the section 
corresponding to $L$. For instance, system (\ref{E:FuchsianSystem4})
defines the Riccati equation 
\begin{equation}\label{E:FuchsianRiccati4}
\frac{dy}{dx}=-b(x)y^2-2a(x)y+c(x)
\end{equation}
on the trivial $\P^1$-bundle by setting $(1:y)=(y_1:y_2)$
where $Y=\begin{pmatrix}y_1\\ y_2\end{pmatrix}$ and 
$A=\begin{pmatrix}a&b\\ c&-a\end{pmatrix}$.
We preferably consider the associate phase portrait, 
namely the singular holomorphic foliation $\mathcal F_0$ induced 
on the ruled surface $\mathbb F_0=\P^1\times\P^1$,
a Riccati foliation
(see \cite{Brunella} or section \ref{sec:ProjConRiccati}).
Singular points of the foliation are located at the poles
of the Riccati equation. Precisely, under notations (\ref{E:ValeursPropres}),
the singular points are
$$\left\{\begin{matrix}
l_i&=&(-b_i:a_i+\frac{\theta_i}{2})=(a_i-\frac{\theta_i}{2}:c_i)\\
l_i'&=&(-b_i:a_i-\frac{\theta_i}{2})=(a_i+\frac{\theta_i}{2}:c_i)
\end{matrix}\right.$$
From the foliation point of view\footnote{For instance, 
$1/\theta_i$ is the Camacho-Sad index of $\mathcal F_1$
along the fibre $x=i$ at $l_i$, see \cite{Brunella}.}, we say that $l_i$ and $l_i'$
have respective {\bf exponents} $\theta_i$ and $-\theta_i$;
they correspond to the eigenlines of the system respectively
associated to eigenvalues $-\frac{\theta_i}{2}$ and $\frac{\theta_i}{2}$
(mind the sign).
In case $\theta_i=0$, either the singular point of the system
is logarithmic and the $2$ singular points of $\mathcal F_1$
coincide, or the singular point of the system is apparent
and the Riccati foliation is not singular. In this latter case,
we will later introduce an additional parabolic structure.
The section $\sigma$ defined by $y=\infty$ plays the role 
of the cyclic vector; it has $2$ tangencies with the Riccati
foliation, namely at $x=q$ and $x=\infty$ (where $\sigma$
passes through a singular point of the foliation),
see bottom of figure \ref{figure:2}.

In this picture, the foliation $\mathcal F_0$ is regular, 
transversal to the
$\P^1$-fibre over a generic point $x$ and is therefore
tranversely projective (see \cite{Croco2}). The foliation
thus induces a projective structure on the section $\sigma$
that projects on the base $\P^1$. It is clear that
the projective structure obtained on (a Zariski open subset of)
$\P^1$ is preserved by a birational bundle transformation
and it is natural to look for the simplest birational model.
In the triple $(\mathbb F_0,\mathcal F_0,\sigma)$ above,
the point $x=\infty$ artificially plays a special role since
we impose by normalization (\ref{E:NormalizationAinfinity})
that the section $\sigma$ passes through the singular point 
$(x,y)=(\infty,\infty)$ of the foliation. 
It is more natural to apply an elementary transformation
at this point and get the following more symetric picture
(see \cite{LorayPereira}):
a Riccati foliation $\mathcal F_1$ on the Hirzebruch ruled 
surface $\mathbb F_1$
having singular points over $0$, $1$, $t$ and $\infty$,
and the section $\sigma$ has now a single tangency 
with the foliation at the point $x=q$. In fact, 
$\sigma:\P^1\to\mathbb F_1$
is the unique negative section (i.e. having $-1$ self-intersection).
The exponents (eigenvalues) of the foliation are now
given by 
$$\boldsymbol{\kappa}=(\kappa_0,\kappa_1,\kappa_t,\kappa_\infty)
:=(\theta_0,\theta_1,\theta_t,\theta_\infty-1).$$
Over each pole $x=i$ of the Riccati equation, the foliation 
$\mathcal F_1$ has $1$ or $2$ singular points depending 
on the exponent $\kappa_i$; when $q=i$,
one of the singular points accidentally lie on $\sigma$,
see the center of figure \ref{figure:2}.

We now add a parabolic structure which is convenient 
for the sequel when we will apply elementary transformations;
it is necessary in order to get a smooth moduli space when 
one of the exponents $\kappa_i$ vanishes.

The Okamoto space of initial conditions $\mathcal M^{\boldsymbol\theta}_{t}$ can be viewed as the moduli space of such Riccati
foliations $\mathcal F_1$. Precisely, we fix exponents
$\boldsymbol{\kappa}$  and parameter $t$, and then consider datas 
$(\mathbb F_1,\mathcal F_1,\sigma,\boldsymbol{l})$
where 
\begin{itemize}
\item $\mathbb F_1$ is the Hirzebruch ruled surface equipped 
with the ruling $\mathbb F_1\to\P^1$,
\item $\mathcal F_1$ is a Riccati foliation on $\mathbb F_1$
regular, transversal to the ruling outside $x=0,1,t,\infty$,
\item over each $i=0,1,t,\infty$, either the fibre $x=i$ is invariant
and there is one singular point for each exponent $\pm\kappa_i$,
or $\kappa_i=0$ and the foliation is regular, transversal to the ruling,
\item $\boldsymbol{l}=(l_0,l_1,l_t,l_\infty)$ where either $l_i$ is the singular point with exponent $\kappa_i$ over $x=i$, or $\kappa_i=0$,
$\mathcal F_i$ is regular over $x=i$ and $l_i$ is any point of the fibre.
\end{itemize}
Such a data exactly corresponds to the projectivization of 
the semistable parabolic connections considered in \cite{IIS}
to construct the moduli space $\mathcal M^{\boldsymbol\theta}_{t}$.

\subsection{Riccati foliations on $\mathbb F_2$}\label{sec:RiccatiF2}
After setting $y=x(x-1)(x-t)u'/u$ in the scalar equation 
(\ref{E:Scalar4+1}), we obtain the Riccati equation
\begin{equation}\label{E:RiccatiF2}
{dy\over dx}=-{y^2\over x(x-1)(x-t)}+
\left({\kappa_0\over x}+{\kappa_1\over x-1}+{\kappa_t\over x-t}+{1\over x-q}\right) y
\end{equation}
$$-{q(q-1)(q-t)p\over x-q}-q(q-1)p+t(t-1)H-\rho(\kappa_\infty+\rho)(x-t).$$
Recall that $\kappa_0+\kappa_1+\kappa_t+\kappa_\infty+2\rho=1$
and $H$ is defined by formula (\ref{E:Hamiltonien}).
Equation (\ref{E:RiccatiF2}) defines a Riccati foliation $\mathcal F_2$ that naturally compactifies on the Hirzebruch surface $\mathbb F_2$
defined by two charts
$$(x,y)\in\left(\P^1\setminus\{\infty\}\right)\times\P^1\ \ \ \text{and}\ \ \ 
(x,\tilde y)\in\left(\P^1\setminus\{0\}\right)\times\P^1$$
with transition map $\tilde y=y/x^2$. The singular points of the foliation
lie on the $5$ fibres
$$x=0,\ \ \ 1,\ \ \ t,\ \ \ q\ \ \ \text{and}\ \ \ \infty$$ 
with respective exponents (up to a sign)
$$\kappa_0,\ \ \ \kappa_1,\ \ \ \kappa_t,\ \ \ \kappa_q=1\ \ \ \text{and}\ \ \ \kappa_\infty.$$
Precisely, the singular points in the first chart $(x,y)$ are given by
\begin{equation}\label{E:SingPointsF2}
\left\{\begin{matrix}s_0=(0,0)\hfill\\ s'_0=(0,t\kappa_0)\end{matrix}\right.\ \ \ 
\left\{\begin{matrix}s_1=(1,0)\hfill\\ s'_1=(1,(1-t)\kappa_1)\end{matrix}\right.\ \ \ 
\left\{\begin{matrix}s_t=(t,0)\hfill\\ s'_t=(t,t(t-1)\kappa_t)\end{matrix}\right.\end{equation}
$$\text{and}\ \ \ 
\left\{\begin{matrix}s_q=(q,\boldsymbol p)\hfill\\ s'_q=(q,\infty)\end{matrix}\right.$$
Here, $s_i$ (resp. $s'_i$) is the singular point with exponent $\kappa_i$ (resp. $-\kappa_i$) and $\boldsymbol p=q(q-1)(q-t)p$.
At $x=\infty$, the singular points are given in chart $(x,\tilde y)$
by
\begin{equation}\label{E:SingPointsF2Infty}
\left\{\begin{matrix}s_\infty=(\infty,-\rho)\hfill\\ s'_\infty=(\infty,-\rho-\kappa_\infty)\end{matrix}\right.
\end{equation}
The ``cyclic vector'' is given by the section $\sigma$ defined
in charts by $y=\infty$ and $\tilde y=\infty$ respectively.

\begin{figure}[htbp]
\begin{center}

\input{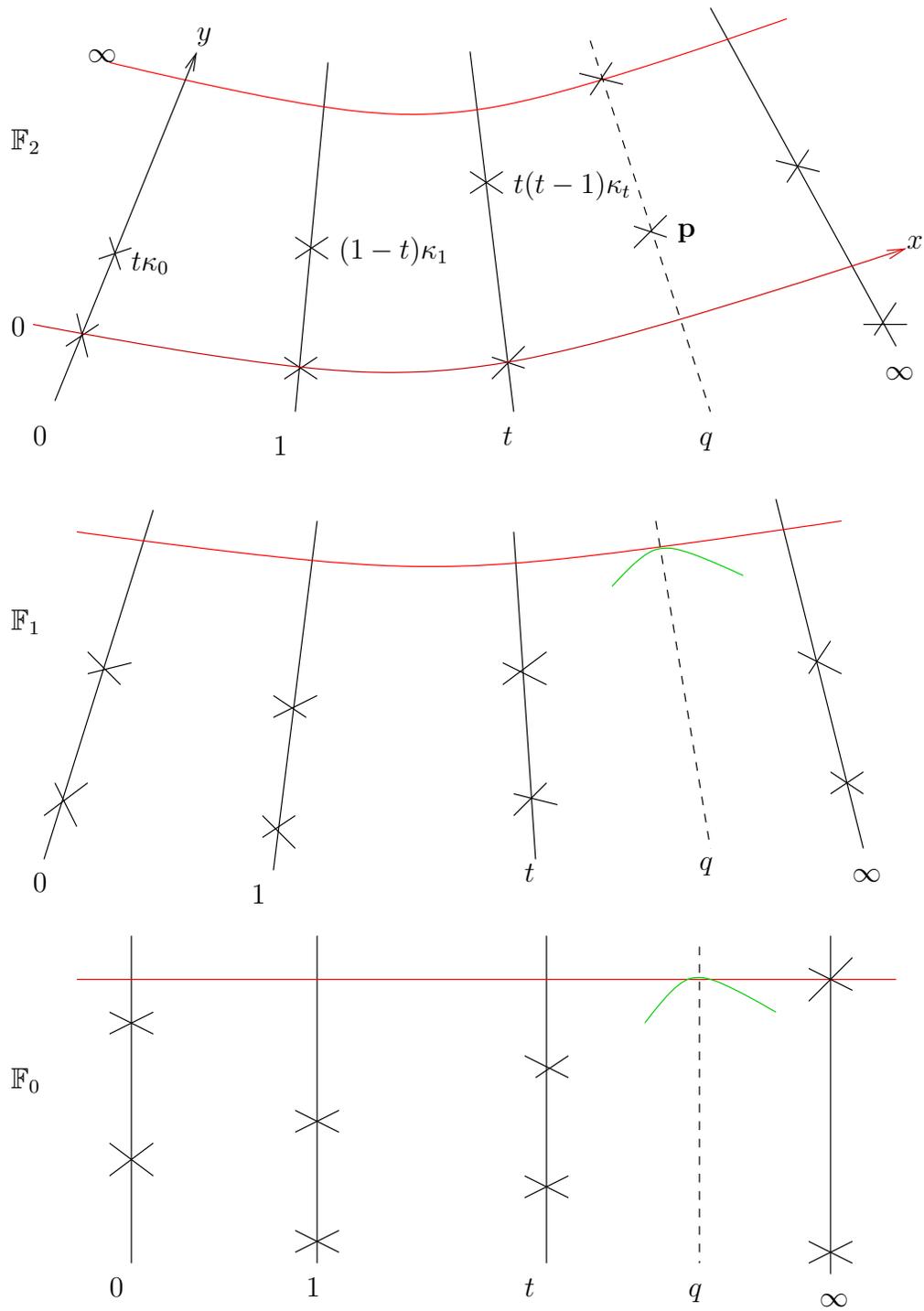}
 
\caption{From $\mathbb F_2$ to $\mathbb F_0$}
\label{figure:2}
\end{center}
\end{figure}

\subsection{Riccati foliations on $\mathbb F_1$}\label{sec:RiccatiF1}
After an elementary transformation with center the nodal singular point $s_q=(q,\boldsymbol p)$ of $\mathcal F_0$,
we obtain a Riccati foliation $\mathcal F_1$ on the Hirzebruch surface
$\mathbb F_1$ with poles $x=0$, $1$, $t$ and $\infty$
with respective exponents $\kappa_0$, $\kappa_1$, $\kappa_t$ and $\kappa_\infty$ like in \ref{sec:ProjStrucRiccati}.
The apparent singular point has disappeared.

If we define $\mathbb F_1$ by usual charts
$$(x,y)\in\left(\P^1\setminus\{\infty\}\right)\times\P^1\ \ \ \text{and}\ \ \ 
(x,\tilde y)\in\left(\P^1\setminus\{0\}\right)\times\P^1$$
with transition map $\tilde y=y/x$, the negative section $\sigma$
is respectively given by $y=\infty$ and $\tilde y=\infty$, and 
the Riccati foliation $\mathcal F_1$ induced by 
\begin{equation}\label{E:RiccatiF1}
{dy\over dx}=\frac{(qy-\boldsymbol p)(qy-\boldsymbol p+t\kappa_0)}{tqx}
\end{equation}
$$-\frac{((q-1)y-\boldsymbol p)((q-1)y-\boldsymbol p+(1-t)\kappa_1)}{(t-1)(q-1)(x-1)}$$
$$+\frac{((q-t)y-\boldsymbol p)((q-t)y-\boldsymbol p+t(t-1)\kappa_t)}{t(t-1)(q-t)(x-t)}
-\rho(\kappa_\infty+\rho).$$
This equation is deduced from (\ref{E:RiccatiF1}) by setting
$y:=(x-q)y+\boldsymbol p$.
The singular points in the first chart $(x,y)$ are now given by
\begin{equation}\label{E:SingPointsF1}
\left\{\begin{matrix}s_0=(0,\frac{\boldsymbol p}{q})\hfill\\ s'_0=(0,\frac{\boldsymbol p-t\kappa_0}{q})\end{matrix}\right.\ \ \ 
\left\{\begin{matrix}s_1=(1,\frac{\boldsymbol p}{q-1})\hfill\\ s'_1=(1,\frac{\boldsymbol p-(1-t)\kappa_1}{q-1})\end{matrix}\right.\ \ \ 
\left\{\begin{matrix}s_t=(t,\frac{\boldsymbol p}{q-t})\hfill\\ s'_t=(t,\frac{\boldsymbol p-t(t-1)\kappa_t}{q-t})\end{matrix}\right.\end{equation}
and the singular points at $x=\infty$ are given in chart $(x,\tilde y)$ by
\begin{equation}
\left\{\begin{matrix}s_\infty=(\infty,-\rho)\hfill\\ s'_\infty=(\infty,-\rho-\kappa_\infty)\end{matrix}\right.
\end{equation}
(again, $s_i$ has exponent $\kappa_i$).

\subsection{Riccati foliations on $\mathbb F_0=\P^1\times\P^1$}\label{sec:RiccatiF0}
Finally, after an ultimate elementary transformation with center the singular point
$s_\infty'$ of $\mathcal F_1$, we obtain a Riccati foliation $\mathcal F_0$
with eigenvalues
$$\boldsymbol\theta=(\theta_0,\theta_1,\theta_t,\theta_\infty)=(\kappa_0,\kappa_1,\kappa_t,\kappa_\infty+1).$$
The underlying ruled surface depends on the relative position of the singular point $s'_\infty$  of $\mathcal F_1$ with respect to the
negative section $\sigma_{-1}$: the Riccati foliation $\mathcal F_0$
is defined on 
\begin{itemize}
\item $\mathbb F_0$ when $s'_\infty\not\in\sigma_{-1}$ (generic case),
\item $\mathbb F_2$ when $s'_\infty\in\sigma_{-1}$ (codimension $1$ case).
\end{itemize}
In particular, $q=\infty$ in the latter case.

When $q\not=\infty$, the Riccati equation defining the foliation 
$\mathcal F_0$ on $\mathbb F_0$ can be deduced from
(\ref{E:RiccatiF1}) by setting $y:=y-(\kappa_\infty+\rho)$.
We obtain equation
\begin{equation}\label{E:RiccatiF0}
{dy\over dx}=\frac{(qy-\boldsymbol p)(qy-\boldsymbol p+t\kappa_0)}{tqx}
\end{equation}
$$-\frac{((q-1)(y-\rho-\kappa_\infty)-\boldsymbol p)((q-1)(y-\rho-\kappa_\infty)-\boldsymbol p+(1-t)\kappa_1)}{(t-1)(q-1)(x-1)}$$
$$+\frac{((q-t)(y-t(\kappa_\infty+\rho))-\boldsymbol p)((q-t)(y-t(\kappa_\infty+\rho))-\boldsymbol p+t(t-1)\kappa_t)}{t(t-1)(q-t)(x-t)}.$$
This is precisely the projectivization of the fuchsian system
defined by (\ref{E:pqcoeffAi}). The singular points $s_i$
with eigenvalues $\kappa_i$ correspond to the parabolic
structure $\boldsymbol{l}$ defined by formula (\ref{E:Formuleli})
\begin{equation}\label{E:SingPointsF0}
\left\{\begin{matrix}
s_0=(0,\frac{\boldsymbol p}{q})\hfill\\
s_1=(1,\frac{\boldsymbol p}{q-1}+\rho+\kappa_\infty)\hfill\\
s_t=(t,\frac{\boldsymbol p}{q-t}+(\rho+\kappa_\infty)t)\\
s_\infty=(\infty,\infty)\hfill
\end{matrix}\right.\ \ \ 
\left\{\begin{matrix}
s_0'=(0,\frac{\boldsymbol p-t\kappa_0}{q})\hfill\\
s_1'=(1,\frac{\boldsymbol p-(1-t)\kappa_1}{q-1}+\rho+\kappa_\infty)\hfill\\
s_t'=(t,\frac{\boldsymbol p-t(t-1)\kappa_t}{q-t}+(\rho+\kappa_\infty)t)
\end{matrix}\right.
\end{equation}
and
$$s'_\infty=-\frac{\boldsymbol p(\boldsymbol p-t\kappa_0)}{tq(\kappa_\infty+1)}
+\frac{\boldsymbol p(\boldsymbol p-(1-t)\kappa_1)}{(t-1)(q-1)(\kappa_\infty+1)}
-\frac{\boldsymbol p(\boldsymbol p-t(t-1)\kappa_t)}{t(t-1)(q-t)(\kappa_\infty+1)}$$
\begin{equation}\label{E:SingPointsF0infty}
-\frac{\kappa_\infty+\rho}{\kappa_\infty+1}
\left((\kappa_\infty+\rho)(q-t-1)-\kappa_1-t\kappa_t\right)
\end{equation}

\subsection{The moduli space $\mathcal M^{\boldsymbol\kappa}_t$}\label{sec:ModuliSpace}
To each $(p,q)\in\C\times\left(\P^1\setminus\{0,1,t,\infty\}\right)$, we have associated a Riccati foliation $\mathcal F_1$
on the Hirzebruch surface $\mathbb F_1$ with poles
$0$, $1$, $t$ and $\infty$ and exponents $\boldsymbol{\kappa}$.

Conversely, given such a Riccati foliation $\mathcal F_1$ 
on $\mathbb F_1$ with poles at $0$, $1$, $t$ and $\infty$, 
the unique tangency $x=q$ between the negative section
$\sigma_{-1}:\P^1\to\mathbb F_1$ and the foliation 
$\mathcal F_1$ defines $q\in\P^1$ uniquely.
Now, if $q\not=0,1,t,\infty$, apply an elementary transformation
at this tangency point to define a Riccati foliation $\mathcal F_2$
on $\mathbb F_2$. There is a unique section $\sigma_2:\P^1\to\mathbb F_2$ having self-intersection $+2$ and passing through 
the singular points $s_0$, $s_1$ and $s_t$ of $\mathcal F_2$;
choose a chart $(x,y)\in\C\times\P^1$ such 
that the section $\sigma_2$ and the negative one, $\sigma_{-2}$,
are respectively given by $y=0$ and $y=\infty$. In fact, 
$y$ is uniquely defined if we normalize the $y^2$-coefficient
of the Riccati equation like in formula (\ref{E:RiccatiF2}).
Then, we observe that all singular points $s_i$ and $s'_i$
for $i=0,1,t,\infty$ only depend on parameters $\boldsymbol{\kappa}$
and $t$; only the singular point $s_q$, given by $(x,y)=(q,\boldsymbol p)$
depend on the particular foliation $\mathcal F_1$. Define $p$ by
$\boldsymbol p=q(q-1)(q-t)p$. We note that {\it the Riccati foliation $\mathcal F_1$ is characterized by the position of the nodal singular point of $\mathcal F_2$ on the Hirzebruch surface $\mathbb F_2$}
(after normalization above). 

\begin{theo}[\cite{IIS}]\label{thm:IISModuliSpace}
Fix parameters
$$\boldsymbol\kappa=(\kappa_0,\kappa_1,\kappa_t,\kappa_\infty)$$
and consider Riccati foliations on the Hirzebruch surface 
$\mathbb F_1$ having at most simple poles 
at $0$, $1$, $t$ and $\infty$, with exponents $\pm\kappa_i$ 
over $x=i$ such that the negative section $\sigma_{-1}$ of $\mathbb F_1$ is not $\mathcal F$-invariant (semi-stability condition). 
Add moreover a parabolic structure 
$\boldsymbol l=(l_0,l_1,l_t,l_\infty)$ which consists over each 
$i=0,1,t,\infty$ of the data
\begin{itemize} 
\item either of the singular point $l_i\in\mathbb F_1$ 
with exponent $\kappa_i$,
\item or, when $\kappa_i=0$ and there is actually no singular point,
of any point of the fibre $x=i$ (the only case where the parabolic 
structure is relevant, i.e. not determined by the foliation itself).
\end{itemize}
The moduli space $\mathcal M^{\boldsymbol\kappa}_t$ of such pairs 
$(\mathcal F,\boldsymbol l)$ up to 
bundle automorphisms is a quasi-projective rational surface 
that can be described as follows. Start with the Hirzebruch surface
$\mathbb F_2$ and denote by $Z_t$ the reduced divisor which consists of the $4$ fibres $x=i$ together with the negative section 
$\sigma_{-2}$. Consider the $8$ points $s_i$ and $s_i'$ defined
by 
\begin{equation}
\left\{\begin{matrix}s_0=(0,0)\hfill\\ s'_0=(0,t\kappa_0)\end{matrix}\right.\ \ \ 
\left\{\begin{matrix}s_1=(1,0)\hfill\\ s'_1=(1,(1-t)\kappa_1)\end{matrix}\right.\ \ \ 
\left\{\begin{matrix}s_t=(t,0)\hfill\\ s'_t=(t,t(t-1)\kappa_t)\end{matrix}\right.\end{equation}
in the first chart $(x,y)$, and
\begin{equation}
\left\{\begin{matrix}s_\infty=(\infty,-\rho)\hfill\\ s'_\infty=(\infty,-\rho-\kappa_\infty)\end{matrix}\right.
\end{equation}
in the chart $(x,\tilde y=\frac{y}{x^2})$. For $i=0,1,t,\infty$:
\begin{itemize} 
\item if $\kappa_i\not=0$, blow-up the two points $s_i$ and $s_i'$,
\item if $\kappa_i=0$, blow-up the single point $s_i=s_i'$ and then 
blow-up the intersection point between the exceptional divisor and the strict transform of the fibre $x=i$.
\end{itemize}
Denote by $\overline{\mathcal M^{\boldsymbol\kappa}_t}$ the 
$8$-point-blowing-up of $\mathbb F_2$ like above. The
moduli space is $\mathcal M_t^{\boldsymbol\kappa}$ is the complement
in $\overline{\mathcal M_t^{\boldsymbol\kappa}}$ of the strict
transform of the divisor $Z_t$.
\end{theo}

The $8$ points $s_i$ and $s_i'$ of the statement are nothing
but the singular points of the foliation on $\mathbb F_2$
described in section \ref{sec:RiccatiF2}.
When the nodal singular point $(x,y)=(q,\boldsymbol p)$ tends
to $s_i$ (resp. $s_i'$), the corresponding foliation on 
$\mathbb F_1$ (see section \ref{sec:RiccatiF1}) has its
singular point $s_i'$ (resp. $s_i$) tending to the negative
section $\sigma_{-1}$. The limit depend on the way the 
nodal point tends to $s_i$ or $s_i'$. When $\kappa_i\not=0$,
the fibre of $\mathcal M_t^{\boldsymbol\kappa}$ over $x=i$
consists in two disjoint copies $S_i$ and $S_i'$ of the affine line $\C$,
over $s_i$ and $s_i'$  respectively: they stand for the moduli 
space of those foliations $\mathcal F$ whose singular point 
$s_i'$ (resp. $s_i$) lie on the negative section $\sigma_{-1}$.
When $\kappa_i=0$, the fibre of $\mathcal M_t^{\boldsymbol\kappa}$ over $x=i$ consists in the union of an affine line $S_i\simeq\mathbb F$
and a projective line $S_i'\simeq\P^1$ that intersect
transversely at one point, and project to $s_i=s_i'$. The component
$S_i$ stands for those foliations for which the point $s_i$ lies
on the negative section $\sigma_{-1}$; the compact component
$S_i'$ stands for those parabolic foliations $(\mathcal F,\boldsymbol l)$
which have actually no pole over $x=i$. If we neglect the parabolic
structure, then the rational curve $S_i'$ blows down to a quadratic
singular point. 

\begin{proof}[Sketch of the proof of Theorem 
\ref{thm:IISModuliSpace}]
Let $\mathcal F$ be a Riccati foliation on $\mathbb F_1$
having at most simple poles over $x=0$, $1$, $t$ and $\infty$, such that the
negative section $\sigma_{-1}$ is not invariant (semi-stability).
In standart the chart $(x,y)$, the foliation is defined by
$$\frac{dy}{dx}=\frac{(a_1x+a_0)y^2+(b_2x^2+b_1x+b_0)y+(c_3x^3+c_2x^2+c_1x+c_0)}{x(x-1)(x-t)}$$
with $a_k,b_k,c_k\in\C$, $a_0$ and $a_1$ not vanishing simultaneously. Bundle automorphisms are given by changes of coordinate of the form $y:=ay+bx+c$, $a,b,c\in\C$, $a\not=0$.
The single tangency between $\mathcal F$ and the negative section $\sigma_{-1}$ is given by $x=q:=-\frac{a_0}{a_1}$. Using change of coordinate $y:=ay$, one may normalize 
$$\text{either}\ \ \ a_1x+a_0=x-q,\ \ \ \text{or}\ \ \ a_1x+a_0=\tilde qx-1,$$
where $\tilde q=\frac{1}{q}$. Let us first assume that $q\not=\infty$
so that we may assume $a_1x+a_0=x-q$. Using a change of coordinate
of the form $y:=y+bx+c$ we may further assume $b_1=b_2=0$.
By the way, assuming $q\not=\infty$, we get a unique normal form
\begin{equation}\label{eq:NormalFormF1}
\frac{dy}{dx}=\frac{(x-q)y^2+b_0y}{x(x-1)(x-t)}+\frac{c_0}{tx}+\frac{c_1}{(1-t)(x-1)}+\frac{c_t}{t(t-1)(x-t)}+c_\infty
\end{equation}
From formula (\ref{E:RiccatiF1}), we get
$$b_0=-2\boldsymbol p+(q-1)(q-t)\kappa_0+q(q-t)\kappa1+q(q-1)\kappa_t.$$
The residue at $x=0$ is defined by 
\begin{equation}\label{Eq:ResF1Gen0}
{dy}=\frac{-qy^2+b_0y+c_0}{t}\frac{dx}{x}+(\text{holomorphic at }x=0)
\end{equation}
and the exponents $\pm\kappa_0$ by the discriminant 
$$\Delta_0:=\frac{b_0^2+4qc_0}{t^2}=\kappa_0^2.$$
Similarly, we get
$$\Delta_1:=\frac{b_0^2+4(q-1)c_1}{(t-1)^2}=\kappa_1^2,\ \ \ 
\Delta_t:=\frac{b_0^2+4(q-t)c_t}{t^2(t-1)^2}=\kappa_t^2,$$
$$\text{and}\ \ \ \Delta_\infty:=1-4c_\infty=\kappa_\infty^2.$$
Once parameter $\boldsymbol{\kappa}$ is fixed, one can determine
uniquely $c_0$, $c_1$, $c_t$ and $c_\infty$ in function of 
$\boldsymbol\kappa$, $q$ and $b_0$ (i.e. $\boldsymbol p$) provided
that $q\not=0,1,t,\infty$. At the neighborhood of $q=0$, 
$c_1$, $c_t$ and $c_\infty$ are still determined as functions
of $\boldsymbol\kappa$, $q$ and $b_0$ and the moduli space 
of such foliations is locally isomorphic to the surface 
$$\{(q,b_0,c_0)\in\C^3,\ b_0^2+2qc_0=(t\kappa_0)^2\}.$$
We promptly see that the moduli space consists, over $q=0$, 
of an affine line parametrized by $c_0$ over each point $b_0=\pm t\kappa_0$. 

When $\kappa_0\not=0$, the graph 
$$c_0=-\frac{(b_0-t\kappa_0)(b_0+t\kappa_0)}{4q}$$
is clearly obtained by blowing-up the two points 
and then deleting the level $c_0=\infty$, 
i.e. the strict transform of $q=0$. From the residue 
(\ref{Eq:ResF1Gen0}), we deduce that are equivalent
over $q=0$:
\begin{itemize}
\item $q=0$ and $b_0=t\kappa_0$ (resp. $b_0=-t\kappa_0$),
\item $q=0$ and $\boldsymbol p=0$ (resp. $\boldsymbol p=t\kappa_0)$,
\item the corresponding foliation $\mathcal F_1$
has its singular point $s_0'$ (resp. $s_0$) lying
on the negative section $y=\infty$.
\end{itemize}
When $\kappa_0=0$, the graph
$$c_0=-\frac{b_0^2}{4q}$$
is now obtained after two blowing-up and 
the affine line parametrized by $c_0$ stands 
for those foliations $\mathcal F_1$ whose singular point
$s_0=s_0'$ lies on the negative section $y=\infty$.
The surface equation $b_0^2+4qc_0=0$ has a quadratic
singular point at $(q,b_0,c_0)=(0,0,0)$ that corresponds 
to the case where the residue (\ref{Eq:ResF1Gen0}) vanishes,
i.e. $\mathcal F_1$ has actually no singular point at $x=0$. 
The parabolic data at $x=0$ provides the desingularization
of the surface. Indeed, the moduli of pairs $(\mathcal F_1,\boldsymbol l)$
is locally parametrized, at the neighborhood of $q=0$,
by 
$$\{(q,b_0,c_0,s_0)\in\C^3\times\P^1,\ b_0^2+2qc_0=0,\ 2q s_0=b_0\}$$
(or we should better write $2q u_0=b_0 v_0$ where $(u_0:v_0)=s_0$).
The parabolic data $s_0=\frac{b_0}{2q}$ parametrizes
the exceptional divisor $S_0'$ (thus providing the blowing up).
The intersection $S_0\cap S_0'$ between the two components,
given by $(q,b_0,c_0,s_0)=(0,0,0,\infty)$, is the foliation without
singular point at $x=0$ whose parabolic structure lies on the negative section $y=\infty$.

The study of $q=1,t,\infty$ is similar except that for $q=\infty$, we have to 
choose the alternate normalization
$$\frac{dy}{dx}=\frac{(\tilde qx-1)y^2+\tilde b_2x^2y}{x(x-1)(x-t)}+\frac{\tilde c_0}{x}+\frac{\tilde c_1}{x-1}+\frac{\tilde c_t}{x-t}+\tilde c_\infty$$
where $(\tilde q,\tilde b_2)=(\frac{1}{q},\frac{b_0}{q^2})$ is the other
chart of $\mathbb F_2$.
\end{proof}

The deformation $t\mapsto \mathcal M^{\boldsymbol\kappa}_t$
is analytically trivial (but not algebraically), and the trivialization
is given by the Painlev\'e flow (see \cite{STT,IIS}). The good 
phase space of the Painlev\'e VI equation is the (locally analytically
trivial) fibration 
$$t:\mathcal M^{\boldsymbol\kappa}\to\P^1\setminus\{0,1,\infty\}.$$
The map $q:\mathcal M^{\boldsymbol\kappa}_t\to\P^1$
is regular (whenever no $\kappa_i=0$) and gives $\mathcal M^{\boldsymbol\kappa}_t$ (for $t$ fixed) a structure of affine 
$\mathbb A^1$-bundle with double fibers over $0,1,t,\infty$. Finally, we note that formula (\ref{E:RiccatiF1}) defines 
an explicit section (universal Riccati foliation)
$$(t,q,\boldsymbol p)\mapsto (\mathbb F_1,\mathcal F_1)$$
over $\mathcal M^{\boldsymbol\kappa}\setminus\{q=0,1,t,\infty\}$.

\subsection{The moduli space $\mathcal M^{\boldsymbol\theta}_t$ of $\mathrm{sl}(2,\C)$-connections}\label{sec:ModuliSpaceSL2}

From section \ref{sec:RiccatiF0}, we get an isomorphism
$$ \mathrm{elm}_{s_\infty'}\ :\ \mathcal M^{\boldsymbol\kappa}_t\stackrel{\sim}{\longrightarrow} \mathcal M^{\boldsymbol\theta}_t$$
from the previous moduli space, to the moduli space $\mathcal M^{\boldsymbol\theta}_t$ of $\mathrm{sl}(2,\C)$-connections
$(E,\nabla)$ having exponent 
$$\boldsymbol{\theta}=(\theta_0,\theta_1,\theta_t,\theta_\infty):=(\kappa_0,\kappa_1,\kappa_t,\kappa_\infty+1).$$
Along this moduli space, recall that the underlying vector bundle $E$ is either trivial, or $\mathcal O(-1)\oplus\mathcal O(1)$;
it is actually trivial on a Zariski open set of $\mathcal M^{\boldsymbol\theta}$.
The locus of the non trivial bundle is Malgrange's Theta Divisor $\Theta\subset\mathcal M^{\boldsymbol\theta}$ defined
for fixed $t$ by the exceptional divisor $S_\infty$ obtained after blowing-up $s_\infty$. Indeed, this corresponds to those 
foliations $\mathcal F_1$ for which $s_\infty'$ lie in the negative section $\sigma_{-1}$; applying $\mathrm{elm}_{s_\infty'}$
gives a foliation $\mathcal F_0$ on the Hirzebruch surface $\mathbb F_2=\mathbb P(\mathcal O(-1)\oplus\mathcal O(1))$.
In fact, we get $\mathcal F_0=\mathcal F_2$ in this case: we are back to the foliation of section \ref{sec:RiccatiF2}
where $s_q\to s_\infty$. Precisely, if we set 
$$p=\frac{\lambda-\rho q}{(q-1)(q-t)}\ \ \ \text{and let}\ \ \ q\to\infty$$
in formula (\ref{E:Scalar4+1}), we then get Heun equation
$$g(x)=\frac{\rho(\rho+\kappa_\infty+1)x-\lambda\kappa_\infty-\rho((t+1)\rho+\kappa_1+t\kappa_t)}{x(x-1)(x-t)}.$$
For our special parameters $\boldsymbol{\theta}=(\frac{1}{2},\frac{1}{2},\frac{1}{2},\frac{1}{2}+\frac{\vartheta}{2})$,
we get Lam\'e equation (\ref{E:LameEquationLegendre}) with 
$n=\frac{\vartheta}{2}$ and $c=2\lambda(\vartheta-1)+(t+1)\frac{\vartheta}{2}\left(\frac{\vartheta}{2}-1\right)$.

Finally, the open set
$\mathcal M^{\boldsymbol\theta}-\Theta$ (resp. the closed subset $\Theta$) may be viewed 
as the moduli space of Riccati foliations $\mathcal F_0$
on $\mathbb F_0$ (resp. $\mathbb F_2$) having simple poles $(0,1,t,\infty)$
and exponents $(\theta_0,\theta_1,\theta_t,\theta_\infty)$ excluding those for 
which the section passing through $s_\infty$ (resp. the exceptional section) 
is totally $\mathcal F_0$ invariant.

\subsection{Okamoto symetries}\label{sec:OkaSym}
There are many birational transformations
$$(\boldsymbol\kappa,t,p,q)\mapsto(\tilde{\boldsymbol\kappa},\tilde{t},\tilde{p},\tilde{q})$$
that induce biregular diffeomorphisms between moduli spaces
$$\mathcal M^{\boldsymbol\kappa}\to\mathcal M^{\tilde{\boldsymbol\kappa}}$$
equivariant with the projection $t$. They are studied in \cite{Okamoto2}: some of them are classical, known as Schlesinger 
transformations, arising from geometrical transformations on 
connections (resp. Riccati foliations); together with a strange extra symetry, they generate
the full group of Okamoto symetries.

\subsubsection{Change of signs}
First of all, we can change the ``spin structure'', i.e. the signs of 
parameters $\pm\kappa_i$. This does not change neither the Riccati
foliation, nor the coefficients $t$, $q$, $b_0$ of the normal form 
(\ref{eq:NormalFormF1}), but the variable $p$ is modified as follows:

$$(-,+,+,+):\ \ \ \left\{\begin{matrix}(\kappa_0,\kappa_1,\kappa_t,\kappa_\infty)&\mapsto&(-\kappa_0,\kappa_1,\kappa_t,\kappa_\infty)\\
t&\mapsto& t\\
(q,p)&\mapsto&(q,p-\frac{\kappa_0}{q})\end{matrix}\right.$$

$$(+,-,+,+):\ \ \ \left\{\begin{matrix}(\kappa_0,\kappa_1,\kappa_t,\kappa_\infty)&\mapsto&(\kappa_0,-\kappa_1,\kappa_t,\kappa_\infty)\\
t&\mapsto& t\\
(q,p)&\mapsto&(q,p-\frac{\kappa_1}{q-1})\end{matrix}\right.$$

$$(+,+,-,+):\ \ \ \left\{\begin{matrix}(\kappa_0,\kappa_1,\kappa_t,\kappa_\infty)&\mapsto&(\kappa_0,\kappa_1,-\kappa_t,\kappa_\infty)\\
t&\mapsto& t\\
(q,p)&\mapsto&(q,p-\frac{\kappa_t}{q-t})\end{matrix}\right.$$

$$(+,+,+,-):\ \ \ \left\{\begin{matrix}(\kappa_0,\kappa_1,\kappa_t,\kappa_\infty)&\mapsto&(\kappa_0,\kappa_1,\kappa_t,-\kappa_\infty)\\
t&\mapsto& t\\
(q,p)&\mapsto&(q,p)\end{matrix}\right.$$

\subsubsection{Permutation of poles}
We can now permute the $4$ poles of the Riccati foliation. 
The order $24$ group is generated by 
$$(01):\ \ \ \left\{\begin{matrix}(\kappa_0,\kappa_1,\kappa_t,\kappa_\infty)&\mapsto&(\kappa_1,\kappa_0,\kappa_t,\kappa_\infty)\\
t&\mapsto& 1-t\\
(q,p)&\mapsto&(1-q,-p)\end{matrix}\right.$$
$$(1t):\ \ \ \left\{\begin{matrix}(\kappa_0,\kappa_1,\kappa_t,\kappa_\infty)&\mapsto&(\kappa_0,\kappa_t,\kappa_1,\kappa_\infty)\\
t&\mapsto& \frac{1}{t}\\
(q,p)&\mapsto&(\frac{q}{t},tp)\end{matrix}\right.$$
$$(0\infty)(1t):\ \ \ \left\{\begin{matrix}(\kappa_0,\kappa_1,\kappa_t,\kappa_\infty)&\mapsto&(\kappa_\infty,\kappa_t,\kappa_1,\kappa_0)\\
t&\mapsto& t\\
(q,p)&\mapsto&(\frac{t}{q},-\frac{q(qp+\rho)}{t})\end{matrix}\right.$$
(the change of variable is given by $\tilde x=1-x$, 
$\frac{x}{t}$ and $\frac{t}{x}$ respectively; we get
 $\tilde b_0=b_0$, $\frac{b_0}{t^2}$ and $\frac{t(q-1)(q-t)-tb_0}{q^2}$  respectively).
 
Together with sign changes, we already get a linear group of order $384$
acting on our moduli space. 

\subsubsection{Elementary transformations}
Let $\mathcal F$ be a Riccati foliation on $\mathbb F_1$
representing a point of $\mathcal M^{\boldsymbol\kappa}_t$.
We can apply an elementary transformation with center one of the singular points of $\mathcal F$, say $s_0$, to obtain a new Riccati foliation with simple poles
over $0$, $1$, $t$ and $\infty$ and shifted parameter $\tilde{\boldsymbol\kappa}=(\kappa_0-1,\kappa_1,\kappa_t,\kappa_\infty)$. The resulting bundle is either the trivial bundle $\mathbb F_0$, or the Hirzebruch surface $\mathbb F_2$.
However, after two (or more generally, an even number of) such 
elementary transformations, we are back to $\mathbb F_1$.
Indeed, since the type $n$ of 
the Hirzebruch surface $\mathbb F_n$ shifts by $\pm 1$ at each elementary transformation, we just have to exclude the possibility, say
$$\begin{matrix}
(\mathbb F_1,\mathcal F)&\dashrightarrow&(\mathbb F_2,\mathcal F')&\dashrightarrow&(\mathbb F_3,\mathcal F'')\\
&\mathrm{elm}_{s_0}&&\mathrm{elm}_{s_\infty}&\end{matrix}$$ 
this would mean that each of the two successive elementary transformations have center lying on the negative section.
In this later case, the negative section $\sigma_{-3}$ of $\mathbb F_3$
is the strict transform of $\sigma_{-1}$: since $\sigma_{-1}$ is not 
$\mathcal F$-invariant, we get that $\sigma_{-3}$ is not 
$\mathcal F''$-invariant, but Proposition \ref{prop:BrunellaSection}
gives a negative tangency, a contradiction.
We have thus defined a biregular transformation
$$\mathrm{elm}_{s_\infty}\circ\mathrm{elm}_{s_0}:\ \ \ \left\{\begin{matrix}(\kappa_0,\kappa_1,\kappa_t,\kappa_\infty)&\mapsto&(\kappa_0-1,\kappa_1,\kappa_t,\kappa_\infty-1)\\
t&\mapsto& t\\
(q,p)&\mapsto&(\tilde q,\tilde p)\end{matrix}\right.$$
We omit the huge formula but we note that $\tilde q$
is given by the unique tangency point between $\mathcal F$
and the unique section $\sigma$ of $\mathbb F_1$ having 
$+1$ self-intersection and passing through $s_0$ and $s_\infty$.

More generally, given any $4$-uple 
$$\boldsymbol{n}=(n_0,n_1,n_t,n_\infty)\in\Z^4,\ \ \ n=n_0+n_1+n_t+n_\infty\in2\Z,$$
we construct a biregular transformation
$$\mathrm{elm}_{\boldsymbol{l}}^{\boldsymbol{n}}:\ \ \ \left\{\begin{matrix}(\kappa_0,\kappa_1,\kappa_t,\kappa_\infty)&\mapsto&(\kappa_0-n_0,\kappa_1-n_1,\kappa_t-n_t,\kappa_\infty-n_\infty)\\
t&\mapsto& t\\
(q,p)&\mapsto&(\tilde q,\tilde p)\end{matrix}\right.$$
where
$$\mathrm{elm}_{\boldsymbol{l}}^{\boldsymbol{n}}=
\mathrm{elm}_{s_0}^{n_0}\circ\mathrm{elm}_{s_1}^{n_1}\circ
\mathrm{elm}_{s_t}^{n_t}\circ\mathrm{elm}_{s_\infty}^{n_\infty}$$
with convention 
$\mathrm{elm}_{s_i}^{n_i}:=\mathrm{elm}_{s_i'}^{-n_i}$
when $n_i<0$.
Like before, there is a unique section $\sigma$ having self-intersection
$\sigma\cdot\sigma=n-1$ and tangency of multiplicity $n_i$
with the foliation at each point $s_i$; the extra tangency 
between $\mathcal F$ and $\sigma$ is at $x=\tilde q$.

We now get an infinite affine group of transformations
that we denote by $H$.

\subsubsection{The Okamoto symetry}
An extra symetry is needed to generate the full group $G$
of biregular transformations decribed in \cite{Okamoto2},
namely
$$\left\{\begin{matrix}(\kappa_0,\kappa_1,\kappa_t,\kappa_\infty)&\mapsto&(\kappa_0+\rho,\kappa_1+\rho,\kappa_t+\rho,\kappa_\infty+\rho)\\
t&\mapsto& t\\
(q,p)&\mapsto&(q+\frac{\rho}{p},p)\end{matrix}\right.$$
(called $s_2$ in \cite{NoumiYamada})
or any of its conjugate. So far, there is no geometric interpretation 
of this symetry as long as we interpret $\mathcal M^{\boldsymbol\kappa}$ as the moduli space of rank $2$ connections (or Riccati
foliations). One has to deal with isomonodromic deformations
of connections of rank $3$ (see \cite{Boalch}) or more (see \cite{NoumiYamada}) in order to derive the full Okamoto
group by natural transformations on the connection.

The conjugate of the Okamoto symetry above by 
the sign change $(+,+,+,-)$, denoted $s_2s_1s_2$ in notations of \cite{NoumiYamada}, is given by:
$$\left\{\begin{matrix}(\kappa_0,\kappa_1,\kappa_t,\kappa_\infty)&\mapsto&(\kappa_0+\rho+\kappa_\infty,\kappa_1+\rho+\kappa_\infty,\kappa_t+\rho+\kappa_\infty,-\rho)\\
t&\mapsto& t\\
(q,p)&\mapsto&(q+\frac{\rho+\kappa_\infty}{p},p)\end{matrix}\right.$$
We recognize in $\tilde q=q+\frac{\rho+\kappa_\infty}{p}$ the Tu invariant of the underlying bundle
of the elliptic pull-back (see Corollary \ref{cor:TuInvariantGeneric}).

\subsection{Special configurations}

For special values of the parameter $\boldsymbol\kappa$,
the moduli space $\mathcal M^{\boldsymbol\kappa}_t$
has rational complete curves (independently on $t$).
They arise from curves on $\mathbb F_2$ avoiding
the negative section and passing through $s_i$
or $s_i'$ over each value $q=0,1,t,\infty$.
They correspond to the locus either of reducible connections,
or connections with a apparent singular point.
It turns out that there are no other complete curve 
in $\mathcal M^{\boldsymbol\kappa}_t$ (see \cite{IIS}),
in particular, no curve for generic values of 
$\boldsymbol\kappa$. Here follow some examples.

When say $\kappa_t=1$ and the foliation $\mathcal F$
has an apparent singular point at $x=t$, then the pole
disappear after one elementary transformation with center $s_t$.
We thus obtain a Hypergeometric type Riccati foliation $\mathcal F_0$
on $\mathbb F_0$ (it cannot be $\mathbb F_1$ by Proposition \ref{prop:BrunellaSection}), that is to say with poles at $0$, $1$
and $\infty$, and exponents $(\kappa_0,\kappa_1,\kappa_\infty)$.
Conversely, given $\mathcal F_0$, we reconstruct a foliation
$\mathbb F$ like above by applying an elementary tranformation
at any point of the fibre $x=t$: this gives us a rational family of
foliations $\mathcal F$ (parametrized by the fibre $x=t$).
The corresponding rational curve $C$ in the moduli space $\mathcal M^{\boldsymbol\kappa}_t$ is given by equation
$$q(q-1)p^2-((q-1)\kappa_0+q\kappa1)p+\frac{(\kappa_0+\kappa_1)^2-\kappa_\infty^2}{4}.$$
This is the unique curve $C$ in $\mathbb F_2$ satisfying:
\begin{itemize}
\item $q:C\to\P^1$ has degree $2$,
\item $C$ does not intersect the negative section $\sigma_{-2}$,
\item $C$ intersects the fibre $q=i$ at both $s_i$ and $s_i'$ 
for $i=0,1,\infty$, 
\item $C$ intersects twice the fibre $q=t$ at $s_t$, 
\item $C$ is singular at $s_t$: it has two smooth branches.
\end{itemize}

\begin{proof}To compute this family in the moduli space, we start from
the Riccati foliation $\mathcal F_0$ that can be normalized to
$$\frac{dy}{dx}=\frac{-y^2-(\kappa_0+x\kappa_\infty)y+cx}{x(x-1)},\ \ \ 
c=\frac{\kappa_1^2-(\kappa_0+\kappa_\infty)^2}{4}$$
(we exclude some reducible cases here) and choose a parabolic
structure $s_t=(t,\lambda)$ over $x=t$. The horizontal section
$y=\lambda$ is sent by the elementary transformation 
$\mathrm{elm}_{s_t}:\mathbb F_0\dashrightarrow\mathbb F_1$ to the negative section: the corresponding value $q=\frac{\lambda^2+\kappa_0\lambda}{c-\kappa_\infty\lambda}$
corresponds to the unique tangency point between $\mathcal F_0$ and $y=\infty$; we already note that the map 
$\lambda\mapsto q(\lambda)$ has degree $2$. Now, 
we observe that the foliation $\mathcal F=\mathrm{elm}_{s_t}\mathcal F_0$ defines a point over $s_0$ (resp. $s_0'$) in the moduli space 
$\mathcal M^{\boldsymbol\kappa}_t$ if, and only if,
the parabolic structure $s_t=(t,\lambda)$ and and the singular point $s_0'$ (resp. $s_0$) of $\mathcal F_0$ lie on the same horizontal
section, namely $y=\lambda$. Indeed, the later section 
is transformed by $\mathrm{elm}_{s_t}$ into the negative section
of $\mathbb F_1$. Thus the curve $C$ passes once through each
point $s_0$ and $s_0'$. The same holds over $q=1$ and $\infty$.
Finally, when the singular point $s_t$ is apparent for $\mathcal F_1$,
it cannot lie on the negative section $\sigma_{-1}$; overwise, after applying $\mathrm{elm}_{s_t}$, we would obtain a hypergeometric foliation on $\mathbb F_2$ having, by Proposition \ref{prop:BrunellaSection},
$-1$ tangencies with the negative section, impossible.
Now we can compute the equation of the curve.
Since $C$ has degree $2$ and does not intersect the negative section of $\mathbb F_2$, 
its equation takes the form
$P^2+A(q)P+B(q)$ where polynomials $A$ and $B$ have
respective degree $2$ and $4$. The fact that $C$ passes through
all points $s_i$ and $s_i'$ but $s_t'$ does not define $C$ 
but a pencil of curves; they are smooth and vertical at $s_t$
(and thus escape from $\mathcal M^{\boldsymbol\kappa}_t$
over this point) except one of them which is singular with normal crossing at $s_t$.
\end{proof}

When $\kappa_t=2$, foliations $\mathcal F$ with apparent singular points at $x=t$ are obtained as follows. Take the hypergeometric
foliation $\mathcal F_0$ on $\mathbb F_1$ with exponents
$(\kappa_0,\kappa_1,\kappa_\infty)$, choose a parabolic structure 
$s_t$ at $x=0$ and apply twice $\mathrm{elm}_{s_t}$. 
Again, we get a rational curve in the moduli space which projects
down to a curve $C\subset\mathbb F_2$ satisfying:
\begin{itemize}
\item $q:C\to\P^1$ has degree $4$,
\item $C$ does not intersect the negative section $\sigma_{-2}$,
\item $C$ intersects the fibre $q=i$ at both $s_i$ and $s_i'$ 
for $i=0,1,\infty$, 
\item $C$ intersects the fibre $q=t$ three times at $s_t$, one time at $s_t'$.
\end{itemize}

\subsection{Bolibrukh-Heu transversality}\label{sec:Viktoria}

A remarkable result of Bolibrukh \cite[page 37]{Bolibrukh} asserts, in our context, that the isomonodromic deformation 
$t\mapsto(E_t,\nabla_t)\in\mathcal M^{\boldsymbol\theta}_t$
of an irreducible $\mathrm{sl}(2,\C)$-connection is ``mostly'' defined on the trivial bundle:

\begin{thm}[Bolibrukh]\label{thm:Bolibrukh}
Let $t\mapsto(E_t,\nabla_t)\in\mathcal M^{\boldsymbol\theta}_t$ be a local isomonodromic deformation.
Then we are in one of the following cases:
\begin{itemize}
\item outside a discrete set of the parameter space $T$, the underlying bundle $E_t$ is trivial,
\item $E_t\equiv\mathcal O(-1)\oplus\mathcal O(1)$ and the destabilizing subsheaf $\mathcal O(1)$ is $\nabla_t$-invariant.
\end{itemize}
\end{thm}

In particular, when $(E_t,\nabla_t)$ is irreducible, we are in the former case. Bolibrukh proved something more general for 
certain logarithmic connections of arbitrary rank on the Riemann sphere. The rank $2$ case has been 
extended in full generality,  to regular/irregular $\mathrm{sl}(2,\C)$-connections, on arbitrary Riemann surfaces,
by V. Heu in \cite{Heu}. We will use the following Corollary

\begin{prop}Let $t\mapsto(E_t,\nabla_t)\in\mathcal M^{\boldsymbol\theta}_t$ be a local isomonodromic deformation
of a $\mathrm{sl}(2,\C)$-connection. Assume that $E_t$ is trivial and two eigenlines $l_i$ and $l_j$ coincide 
along the deformation, $i,j\in\{0,1,t,\infty\}$. Then $(E_t,\nabla_t)$ is reducible: the constant line bundle $L_t\subset E_t$ defined by $l_i=l_j$
is $\nabla_t$-invariant.
\end{prop}

\begin{proof}After applying $\elm_{s_i}\circ\elm_{s_j}$ to the deformation, we get an isomonodromic deformation 
on the Hirzebruch surface $\mathbb F_2$ which is possible only when the connection is reducible by Bolibrukh Theorem \ref{thm:Bolibrukh}.
\end{proof}

\begin{remark}We have the following more general result. Consider an irreducible Riccati foliation $\mathcal F_0$ in $\mathcal M^{\boldsymbol\theta}_t$
defined on the trivial bundle $\P^1\times\P^1$. Sections $\sigma_d$ having self-intersection $d\ge0$ ($d$ even) form a $(d+1)$-dimensional family. 
The smooth curve $\sigma_d$ has exactly $d+2$ tangencies with $\mathcal F_0$ counted with multiplicities. Then the tangency locus
can be totally contained in fibers over $\{0,1,t,\infty\}$ only at isolated points of the parameter space $T$. For $d=0$ we recover the proposition.
For $d=2$ we get for instance all $4$ parabolics cannot lie along a bidegree $(1,1)$ curve along the deformation.
\end{remark}


\section{Lam\'e connections}\label{sec:LameConnection}

The aim of this section is to roughly describe the moduli space
of Lam\'e connections up to biregular bundle transformations 
by means of the Riemann-Hilbert correspondence. 
Here, we fix the elliptic 
curve 
$$ X:\{y^2=x(x-1)(x-t)\}.$$
We will point out
which Lam\'e connections are invariant under the elliptic involution 
$$\sigma\ :\ X\to X\ ;\ (x,y)\mapsto(x,-y).$$
In the next section, we will see that $\sigma$-invariant Lam\'e connections
can be pushed down, via the double cover 
$$\pi:X\to\P^1\ ;\  (x,y)\mapsto x$$
as a logarithmic connection with poles at the $4$
ramification points $i=0,1,t,\infty$.

Let $(E,\nabla)$ be a Lam\'e connection over the elliptic
curve $X$, thus having a simple pole at $\omega_\infty$.
When the exponent $\vartheta$ is not an integer,
the connection may
be reduced to the following matrix form
$$\nabla\ :\ W\mapsto dW-\Omega W,\ \ \ \Omega=\begin{pmatrix}\frac{\vartheta}{2}\frac{dz}{z}&0\\
0&-\frac{\vartheta}{2}\frac{dz}{z}\end{pmatrix}$$
where $z\in(\C,0)$ is any local coordinate of $X$ 
at $\omega_\infty$, and $W\in\C^2$, a convenient 
local holomorphic trivialization of $E$.
On the other hand, when $\vartheta\in\Z$, say $\vartheta=n\in\Z_{\le0}$, 
the pole is {\bf resonant} and the matrix form may be reduced
(by local gauge transformation as above) to
\begin{equation}\label{PoincareDulac}
\text{either}\ \ \ \Omega=\begin{pmatrix}\frac{n}{2}&0\\
0&-\frac{n}{2}\end{pmatrix}\frac{dz}{z},\ \ \ \text{or}\ \ \ 
\Omega=\begin{pmatrix}\frac{n}{2}&z^n\\
0&-\frac{n}{2}\end{pmatrix}\frac{dz}{z}.
\end{equation}
The point $\omega_\infty$ is an {\bf apparent} singular point 
for $\nabla$
in the former case (actually regular when $n=0$), and a
{\bf logarithmic} singular point in the latter case.

The connection $\nabla$ is regular over the affine part
$X^*=X-\{\omega_\infty\}$ of the curve,
and we inherit a {\bf monodromy representation} 
$$\rho\ :\ \pi_1(X^*)\to\mathrm{SL}(2,\C)$$
which is well defined by $(E,\nabla)$ up to $\mathrm{SL}(2,\C)$-conjugacy. Fix a loop $\delta\in\pi_1(X^*)$ going to 
$\omega_\infty$,
turning once around, and coming back to the base point;
then $\rho(\delta)$ is the {\bf local monodromy} of $(E,\nabla)$
{\bf around} $\omega_\infty$ and is 
conjugated to 
$$\begin{pmatrix}e^{i\pi\vartheta}&0\\ 0&e^{-i\pi\vartheta}\end{pmatrix}\ \ \ \left(\text{resp.}\ \pm\begin{pmatrix}1&1\\ 0&1\end{pmatrix}\ \text{in the logarithmic case}\right).$$
All of this obviously does not depend on the choice of the base point
for the fundamental group. We note that the singular point 
$\omega_\infty$ is an apparent singular point if, and only if, 
the local monodromy $\rho(\delta)$ is $\pm I$, 
the center of $\mathrm{SL}(2,\C)$; 
this can occur only when $\vartheta\in\Z^*$.

\subsection{Riemann-Hilbert correspondence}
For each 
exponent
$\vartheta\in\C$, we get an analytic map
$$\begin{matrix}&RH&\\
\left\{\begin{matrix}\text{Lam\'e connections over }X\\ 
\text{with exponent }\vartheta\end{matrix}\right\}/\sim     &\to&
\left\{\begin{matrix}\rho:\pi_1(X^*)\to\mathrm{SL}(2,\C)\\ 
\mathrm{trace}(\rho(\delta))=2\cos(\pi\vartheta)\end{matrix}\right\}/\sim    \\
(E,\nabla)&\mapsto&\rho\end{matrix}$$
which assigns to a Lam\'e connection, up to holomorphic bundle isomorphisms, the corresponding monodromy representation,
up to $\mathrm{SL}(2,\C)$-conjugacy. 
This map is almost one-to-one: it is surjective, and
it is injective in restriction to those connections without
apparent singular point, i.e. such that the local monodromy
is $\rho(\delta)\not=\pm I$. 

When $(E,\nabla)$ has an apparent singular point at $\omega_\infty$,
say $\vartheta=n\in\Z_{>0}$, all horizontal sections have
meromorphic extension at $\omega_\infty$; those holomorphic ones
are contained in a sub-line-bundle of $E$ near $\omega_\infty$,
say $L$; through the normal form (\ref{PoincareDulac}),
$L$ is the constant line bundle generated by 
$\begin{pmatrix}1\\0\end{pmatrix}$. 
Except in very special cases, $L$ does not extends
as a line bundle $L\subset E$ on the whole of $X$: it is only defined
at the neighborhood of $\omega_\infty$. The fiber of $L$ 
at $\omega_\infty$ coincides with the eigenline of the residual
matrix associated to the positive eigenvalue $\frac{n}{2}$.
By the Riemann-Hilbert 
correspondence over the punctured curve $X^*$, any two Lam\'e
connections with the same monodromy representation are conjugated
by a gauge transformation over $X^*$; the conjugacy extends
as a global gauge transformation if, and only if, it conjugates 
the corresponding local line bundles as defined above.
One can restore the injectivity of the Riemann-Hilbert map
in the following way. Consider the monodromy representation $\rho$
as an action of the fundamental group $\pi_1(X^*,p)$
on the space $E_p\simeq\C^2$ of germs of solutions at the base 
point $p$; now, dragging back, by analytic continuation along (half-)$\delta$, the local holomorphic solutions at $\omega_\infty$ until the base point $p$,
we get a one dimensional subspace $L_p\subset E_p$.
In other words, $L_p$ is obtained by analytic continuation of $L$ 
(as a $\nabla$-invariant line bundle) along $\delta$.
The Lam\'e connection (with apparent singular point) is characterized
by the pair 
$$(\rho,L_p)\in\mathrm{Hom}\left(\pi_1(X^*,p),SL(E_p)\right)\times\P(E_p)$$
up to conjugacy:
$$(\rho,L_p)\sim(M^{-1}\rho M,M^{-1}L_p),\ \ \ 
M\in\mathrm{SL}(2,\C).$$
This is a kind of parabolic structure for the space of representations.

\subsection{Fricke moduli space}
Let us first recall how to describe the moduli space 
of representations following Fricke (see \cite{Goldman}).
Fix standart generators $\alpha,\beta\in\pi_1(X^*)$ of the fundamental group so that the commutator $\delta:=[\alpha,\beta]=\alpha\beta\alpha^{-1}\beta^{-1}$ represents a small loop turning once around the puncture as before. We neglect the base point as it will play
no role in our discussion. A representation $\rho$ is determined
by the images of generators 
$$A:=\rho(\alpha)\ \ \ \text{and}\ \ \ B:=\rho(\beta).$$
The ring of polynomial functions on 
$\mathrm{SL}(2,\C)\times\mathrm{SL}(2,\C)$ 
that are invariant under the $\mathrm{SL}(2,\C)$-conjugacy 
action is generated by 
$$a:=\tr(A),\ \ \ b:=\tr(B)\ \ \ \text{and}\ \ \ c:=\tr(AB);$$
for instance, the trace of the commutator $\delta=[\alpha,\beta]$ is given by
$$d:=\tr([A,B])=a^2+b^2+c^2-abc-2.$$
and we have
$$\rho\ \text{is reducible}\ \Leftrightarrow\ a^2+b^2+c^2-abc-2=2.$$
The geometric quotient of $\mathrm{Hom}(\pi_1(X_t^*),\mathrm{SL}(2,\C))$ by the $\mathrm{SL}(2,\C)$-conjugacy 
action identifies with $\C^3$ via the composition
$$\begin{matrix}
\mathrm{Hom}(\pi_1(X^*),\mathrm{SL}(2,\C))&\to&
\mathrm{SL}(2,\C)\times\mathrm{SL}(2,\C)&\to&\C^3\\
\rho&\mapsto&(A,B)&\mapsto&(a,b,c)\end{matrix}$$
Precisely (see \cite{Churchill,Goldman}), when $a^2+b^2+c^2-abc-2\not=2$,
the fibre over $(a,b,c)$
consists in the single $\mathrm{SL}(2,\C)$-conjugacy
class of the irreducible representation defined by
$$A=\begin{pmatrix}a&-1\\ 1&0\end{pmatrix}\ \ \ \text{and}\ \ \ 
B=\begin{pmatrix}0&\gamma^{-1}\\ -\gamma&b\end{pmatrix}\ \ \ 
\text{with}\ \ \ \gamma+\gamma^{-1}=c.$$
This normal form is obtained in any basis of the form $(v,-\gamma B.v)$ where $v$ 
is an eigenvector for the product $A.B$ with eigenvalue $\gamma$;
it only depends on the choice of the root $\gamma$.
The commutator is therefore given by 
$$[A,B]=\begin{pmatrix}
-\frac{2\gamma^2+1}{\gamma^2}&\frac{a-b\gamma}{\gamma^2}\\ 
\frac{a\gamma-b}{\gamma}&\frac{1}{\gamma^2}\end{pmatrix}.
$$
One can check by direct computation that matrices $A$
and $B$ above share a common eigenvector if, and only if,
$d=2$.

\begin{cor}A Lam\'e connection is reducible if, and only if,
$\vartheta\in\Z$.
\end{cor}

The elliptic involution $\sigma:X\to X;(x,y)\mapsto(x,-y)$ acts on our moduli space 
sending the $\mathrm{SL}(2,\C)$-class defined by $(A,B)$ 
onto that one defined by $(A^{-1},B^{-1})$; this may be seen by choosing a fixed point of $\sigma$ as the base point for the fundamental group. One of the key point
of our construction is 

\begin{lem}\label{Lem:SL2Involution}
An $\mathrm{SL}(2,\C)$-class is stabilized
by the elliptic involution $\sigma$ if, and only if, it consists
in either irreducible, or abelian representations. In other words,
for any pair $(A,B)$ generating an irreducible or abelian subgroup
of $\mathrm{SL}(2,\C)$, there exists $M\in\mathrm{SL}(2,\C)$ such that 
$$M^{-1}AM=A^{-1}\ \ \ \text{and}\ \ \ M^{-1}BM=B^{-1}.$$
In the irreducible case, $M$ is unique up to a sign
and $\tr(M)=0$.
\end{lem}

\begin{proof}In $\mathrm{SL}(2,\C)$, we have 
$\tr(A^{-1})=\tr(A)$ and 
$$\tr(A^{-1}B^{-1})=\tr(B^{-1}A^{-1})=\tr((AB)^{-1})=\tr(AB)$$ 
so that the involution acts trivially on the quotient,
i.e. on triples $(a,b,c)$.
Since irreducible $\mathrm{SL}(2,\C)$-classes are 
characterized by their corresponding triple $(a,b,c)$, 
they are $\sigma$-invariant. Another way to show this
is to note that the matrix $M$ as in the statement has
to permute the two eigenvectors of each matrix $A$ and $B$; 
in $\mathrm{PSL}(2,\C)$, we are looking for an element
$\overline{M}$ permuting the corresponding points in $\P^1$, 
sending a quadruple $(a_1,a_2,b_1,b_2)$ to the quadruple
$(a_2,a_1,b_2,b_1)$. But the cross-ratios are the same, 
thus showing the existence of $\overline{M}$; moreover,
$\overline{M}$ is an involution since its square fixes the $4$
points, and therefore $\tr(M)=0$. Degenerate cases where 
$a_1=a_2$ or/and $b_1=b_2$ have to be treated apart;
we omit this discussion.
In the reducible case, we have say $a_1=b_1$ and 
$\overline{M}$ exists if, and only if, $a_2=b_2$, thus implying
abelianity. Finally, $\sigma$ exchanges upper
and lower triangular $\mathrm{SL}(2,\C)$-representations
but stabilizes diagonal ones.
\end{proof}

When the matrices $A$ and $B$ are in the normal form above,
the matrix $M$ of the previous Lemma is given, up to a sign, by 
\begin{equation}\label{InvolMatrix}
M=\begin{pmatrix}
\frac{\gamma^2-1}{2\gamma}&\frac{a-b\gamma}{2\gamma}\\ 
\frac{a\gamma-b}{2}&-\frac{\gamma^2-1}{2\gamma}\end{pmatrix}
\end{equation}  
We resume our discussion in the non resonant case:

\begin{cor}\label{Cor:LameNotZ}Given an elliptic curve $X$ and $\vartheta\not\in\Z$, 
Lam\'e connections on $X$ with exponent $\vartheta$ are in one-to-one correspondence with the points
of the smooth affine hypersurface
$$S_d:=\{(a,b,c)\in\C^3\ ;\ a^2+b^2+c^2-abc-2=d\},\ \ \ d=2\cos(\pi\vartheta);$$
they are irreducible and $\sigma$-invariant.
\end{cor}

\begin{proof}The connection is irreducible 
($\vartheta\not\in\Z$)
and has no apparent singular point. The Riemann-Hilbert is therefore 
injective and assertions directly follow from Lemma \ref{Lem:SL2Involution}.
\end{proof}

\subsection{Resonant cases}

We now complete the picture with those resonant parameters
$\vartheta\in\Z$.

\subsubsection{$\vartheta\in\Z\setminus 2\Z$}
The monodromy representation is irreducible, characterized 
by the corresponding triple $(a,b,c)$.
The local monodromy around $\omega_\infty$ is parabolic, with twice the eigenvalue $-1$, and is given by the commutator 
$$[A,B]=\begin{pmatrix}
a^2+b^2+\gamma^2-abc&\gamma^{-2}(a-b\gamma)\\
a-b\gamma^{-1}&\gamma^{-2}\end{pmatrix},\ \ \ \gamma+\gamma^{-1}=c;$$
we have $\tr([A,B])=a^2+b^2+c^2-abc-2=-2$
and $[A,B]=-I$ precisely when $(a,b,c)=(0,0,0)$,
the unique singular point of the surface. 

\begin{prop}\label{Prop:LameZlog}
Lam\'e connections with exponent 
$\vartheta\in\Z\setminus 2\Z$ 
having a logarithmic singular point
are in one-to-one correspondence 
with the smooth points of the Markov affine hypersurface
$$S_{-2}=\{(a,b,c)\in\C^3\ ;\ a^2+b^2+c^2-abc=0\};$$
they are irreducible and $\sigma$-invariant.
\end{prop}

\begin{proof}The same as for Corollary \ref{Cor:LameNotZ}.\end{proof}

When $(a,b,c)=(0,0,0)$, the image of the monodromy representation 
is the order $8$ dihedral group (i.e. quaternionic)
$$A=\begin{pmatrix}0&-1\\1&0\end{pmatrix}\ \ \ \text{and}\ \ \ 
B=\begin{pmatrix}0&i\\ i&0\end{pmatrix}$$
and the singular point $\omega_\infty$ of the connection
is apparent: $[A,B]=-I$. In this case, the Lam\'e connection 
is not characterized by its monodromy representation 
and we indeed have

\begin{prop}\label{Prop:LameZapp}
Lam\'e connections over the singular point
$(a,b,c)=(0,0,0)$ are in one-to-one correspondence
with $\P^1$. They have the same monodromy 
representation into the order $8$ dihedral subgroup 
of $\mathrm{SL}(2,\C)$,
thus irreducible; all of them are $\sigma$-invariant.
\end{prop}

\begin{proof}The Lam\'e connection is determined 
by its monodromy representation $\rho$, acting on the space 
of solutions $E_p\simeq\C^2$,  and the line 
$L_p\subset E_p$ corresponding to those solutions 
holomorphic at $\omega_\infty$ after analytic continuation
along $\delta$. In other words, the connection is determined 
by a triple 
$$(A,B,L)\in\mathrm{SL}(2,\C)\times\mathrm{SL}(2,\C)\times\P^1;$$
another triple $(\tilde A,\tilde B,\tilde L)$ will represent a connection
gauge equivalent to the initial one if, and only if
$$(\tilde A,\tilde B,\tilde L)=(M^{-1}AM,M^{-1}BM,M^{-1}L),\ \ \ M\in\mathrm{SL}(2,\C).$$
The monodromy representation, being irreducible here, has
centralizor $\pm I$ acting trivially on $\P^1$:
once the monodromy representation $(A,B)$ is fixed,
gauge equivalence classes of connections are in one-to-one
correspondence with $\P^1\supset L$.

One easily check that the action of $\sigma$ on Lam\'e connections
induces the following action on the corresponding triples:
$$\sigma:(A,B,L)\mapsto(A^{-1},B^{-1},(AB)^{-1}L).$$
It turns out that, when $(a,b,c)=(0,0,0)$ (and for instance $\gamma=i$), the matrix $M$ given by Lemma \ref{Lem:SL2Involution} (see (\ref{InvolMatrix})) satisfies
$$M=A\cdot B=\begin{pmatrix}i&0\\ 0&-i\end{pmatrix}$$
and thus conjugates $\sigma(A,B,L)$ to $(A,B,L)$, thus proving
the $\sigma$-invari\-ance of the corresponding connection,
a kind of miracle.
\end{proof}

\begin{remark}In fact, the moduli space of Lam\'e connections with fixed exponent 
$\vartheta\in\Z\setminus 2\Z$ may be viewed,
over the surface $S_{-2}:\{a^2+b^2+c^2-abc=0\}$, as the minimal resolution
obtained after blowing-up once the singular point at $(0,0,0)$:
the exceptional divisor stands for those connections with apparent
singular point considered in Proposition \ref{Prop:LameZapp}.
This will immediately follow
from the similar result obtained in \cite{IIS} for connections
over $\P^1$ after our descent construction; let us
give some direct arguments.
Along the smooth part of the affine surface $S_{-2}$,
one can consider the
one dimensional subspace $L_p\subset E_p$ of solutions holomorphic
at $\omega_\infty$ after analytic continuation along $\delta$: 
one can check from the local model (\ref{PoincareDulac})
in the logarithmic case that $L_p$ coincides with the eigenspace
of the local monodromy $[A,B]$, namely
$$L=\C\cdot\begin{pmatrix}\gamma^2+1\\ -a\gamma^2+b\gamma
\end{pmatrix}.
$$
The exceptional divisor of $S_{-2}$ is given by $a^2+b^2+c^2=0$
in homogeneous coordinates $(a:b:c)$ and can be parametrized by
$$\P^1\hookrightarrow\P^2\ ;\ 
s\mapsto (i(s^2+1):s^2-1:2s).$$
One can easily verify that the line $L_p$ tends to 
$\C\cdot\begin{pmatrix}1\\ s\end{pmatrix}$
when the representation $\rho$ tends to the point $s$
via the parametrization above.
\end{remark}

\subsubsection{$\vartheta\in2\Z$}
In this case, a combination of several non Hausdorff phenomena 
occur for both moduli spaces of representations, and connections. 
The Hausdorff quotient is given by the Cayley affine hypersurface
$$S_{2}=\{(a,b,c)\in\C^3\ ;\ a^2+b^2+c^2-abc=4\}.$$
The singular points of $S_{2}$ are 
$$(a,b,c)=(2,2,2),\ (2,-2,-2),\ (-2,2,-2)\ \ \ \text{and}\ \ \ (-2,-2,2);$$
they play the same role in the sense that they are permuted by 
changing signs of generators:
$$(A,B),\ (A,-B),\ (-A,B)\ \ \ \text{and}\ \ \ (-A,-B).$$

Over a smooth point $(a,b,c)\in S_{2}$, there are exactly 
$3$ distinct $\mathrm{SL}(2,\C)$-conjugacy
classes of representations, namely
$$(A,B)=\left( 
\begin{pmatrix}\alpha&\lambda\\ 0&\alpha^{-1}\end{pmatrix},
\begin{pmatrix}\beta&\mu\\ 0&\beta^{-1}\end{pmatrix}\right),\ \ \ 
\left( 
\begin{pmatrix}\alpha&0\\ \lambda&\alpha^{-1}\end{pmatrix},
\begin{pmatrix}\beta&0\\ \mu&\beta^{-1}\end{pmatrix}\right)$$
$$\text{(genuine upper and lower triangular)}$$
$$\ \ \ \text{and}\ \ \ 
\left( 
\begin{pmatrix}\alpha&0\\ 0&\alpha^{-1}\end{pmatrix},
\begin{pmatrix}\beta&0\\ 0&\beta^{-1}\end{pmatrix}\right)
\ \ \ \text{where}\ \ \ 
\left\{\begin{matrix}\alpha+\alpha^{-1}&=&a\\
\beta+\beta^{-1}&=&b\\
\alpha\beta+(\alpha\beta)^{-1}&=&c\end{matrix}\right.$$
In the triangular cases, $\lambda$ and $\mu$ can be choosen arbitrarily, 
provided that $[A,B]\not=I$, i.e. $b\lambda+a\mu\not=0$.

Each of the two triangular representations correspond to 
a unique Lam\'e connection (with a logarithmic singular point).
They are permuted by $\sigma$ and thus not $\sigma$-invariant.
The moduli space of those triangular connections is a $2$-fold
cover of the smooth part $S_{2}^*$ of $S_{2}$, the two sheets of which 
are permuted around each of the four singular points.

When $\vartheta=0$, the diagonal representation corresponds
to a unique (regular) Lam\'e connection which is $\sigma$-invariant.

When $\vartheta\not=0$, Lam\'e connections
over the diagonal representation
have an apparent singular point: there are exactly $3$ 
equivalence classes corresponding to the following 
choices for the line bundle $L$ (see the proof of 
Proposition \ref{Prop:LameZapp})
$$\P^1\ni L=(1:0),\ (0:1)\ \ \ \text{or}\ \ \ (1:1)$$
(any choice $L=(1:s)$, $s\in\C^*$, is equivalent to $L=(1:1)$).
The involution $\sigma$ permutes the two first connections
while it fixes the ``generic'' third one: in both situations, 
it suffices to choose the matrix
$$M=\begin{pmatrix}0&i\gamma^{-1}\\ i\gamma&0\end{pmatrix}$$ (see again proof of Proposition \ref{Prop:LameZapp}).
By the way, Lam\'e connections with diagonal monodromy split into
another two-fold cover of $S_{2}^*$, with Galois involution $\sigma$,
and a copy of $S_{2}^*$, on which $\sigma$ acts trivially.

Finally, {\it over each smooth point $(a,b,c)\in S_{2}^*$, 
there are exactly $5$ Lam\'e connections 
(resp. $3$ when $\vartheta=0$), only one of which 
is $\sigma$-invariant.}

Consider now a singular point, say $(a,b,c)=(2,2,2)$
(recall that the four singular points play the same role).
Above this point, there are infinitely many distinct 
$\mathrm{SL}(2,\C)$-conjugacy classes in the fibre
defined by parabolic pairs 
$$(A,B)=\left( 
\begin{pmatrix}1&r\\ 0&1\end{pmatrix},
\begin{pmatrix}1&s\\ 0&1\end{pmatrix}\right),
\ \ \ (r:s)\in\P^1$$
and the central one $(A,B)=(I,I)$ (when $(r,s)=(0,0)$).
When $\vartheta=0$, those representations bijectively
correspond to Lam\'e connections that are $\sigma$-invariant.
When $\vartheta\not=0$, then for each parabolic
representation, there are exactly $2$ Lam\'e connections 
given by $L=(1:0)$ 
and $L=(s:1)$ (any $s\in\C$ are equivalent) and one Lam\'e connection with
trivial monodromy;
all of them are $\sigma$-invariant. 

\begin{remark}\rm
When we have an apparent singular point and 
the direction $L$ is fixed by the monodromy,
we get a $\nabla$-invariant line bundle, say $L$ again, 
having positive degree, thus implying unstability of the bundle $E$.
Indeed, in this case we have $\vartheta=\pm 2m$, $m\in\Z_{>0}$ and $\nabla$-horizontal sections in $L$ are holomorphic,
vanishing at the order $m$ at $\omega_\infty$; by Fuchs relation,
$\deg(L)=m$. We do not want to consider this kind of deformations 
in this paper.
\end{remark}

We resume a part of our discussion

\begin{prop}When $\vartheta=2m\in\Z$, all semistable and 
$\sigma$-invariant Lam\'e connections have an apparent singular point
(resp. regular when $m=0$) at $\omega_\infty$ and their monodromy 
data belong to the following list: 
\begin{itemize}
\item $(A,B)=( 
\begin{pmatrix}\alpha&0\\ 0&\alpha^{-1}\end{pmatrix},
\begin{pmatrix}\beta&0\\ 0&\beta^{-1}\end{pmatrix})$ 
with $(\alpha,\beta)\not=(\pm1,\pm1)$\hfill\break
\hfill and $L=\begin{pmatrix}1\\ 1\end{pmatrix}$ when $m\not=0$;
\item $(A,B)=(
\pm\begin{pmatrix}1&s\\ 0&1\end{pmatrix},
\pm\begin{pmatrix}1&t\\ 0&1\end{pmatrix})$ 
with $(s,t)\in\P^1$\hfill\break
and $L=\begin{pmatrix}1\\ 1\end{pmatrix}$ when $m\not=0$.
\end{itemize}
When $m=0$, we also have to add the $4$ connections with monodromy $(A,B)=(\pm I,\pm I)$.
\end{prop}

\subsection{Irreducible Lam\'e connections are elliptic pull-back}\label{sect:IrreducibleLamePullBack}

We now check that $\sigma$-invariant representations are actually coming from representations
of the $4$-punctured sphere via the elliptic cover. Precisely, let us consider the elliptic pull-back
construction of section \ref{sec:MainConstruction} from the monodromy representation point of view.
For a connection $(E,\nabla)\in\mathcal M_{t}^{\boldsymbol\theta}$ with exponents 
$$\boldsymbol{\theta}=(\frac{1}{2},\frac{1}{2},\frac{1}{2},\frac{1}{2}+\frac{\vartheta}{2}),$$
consider the monodromy representation 
$$\pi_1(\P^1\setminus\{0,1,t,\infty\})\to \mathrm{SL}(2,\C).$$
It is defined by matrices $(M_0,M_1,M_t,M_\infty)$ satisfying 
\begin{equation}\label{Eq:IdentMi}
(M_0)^2=(M_1)^2=(M_t)^2=-I\ \ \ \text{and}\ \ \ \tr(M_\infty)=-2\sin(\frac{\pi\vartheta}{2})
\end{equation}
The monodromy of its elliptic pull-back is therefore given by 
$$A=M_0M_1\ \ \ \text{and}\ \ \ B=M_1M_t$$
(see \cite[section 2]{DynChar} for details), the commutator by
$$[A,B]=-M_0(M_\infty)^2M_0^{-1}$$
and we can check that its trace is given by 
$$\tr([A,B])=-\tr((M_\infty)^2)=2-(\tr(M_\infty))^2=2\cos(\pi\vartheta).$$
Clearly, this representation is $\sigma$-invariant since for $M:=\pm M_1$ we get from (\ref{Eq:IdentMi}) that
$$M^{-1}AM=A^{-1}\ \ \ \text{and}\ \ \ M^{-1}BM=B^{-1}.$$
Conversely, let $(A,B)$ defines the monodromy of a $\sigma$-invariant Lam\'e connection:
there is a matrix $M$ conjugating $(A,B)$ to $(A^{-1},B^{-1})$. From the previous sections, 
it is clear that we can assume that $M$ has null trace:
$M^2=-I$. Then it is straightforward to check that $(A,B)$ is the elliptic pull-back of the following represention 
$$M_0=-AM,\ \ \ M_1=M,\ \ \ M_t=-MB\ \ \ \text{and}\ \ \ M_\infty=B^{-1}MA^{-1}.$$
When the monodromy $(A,B)$ is irreducible, 
then $(M_0,M_1,M_t,M_\infty)$ is the unique quadruple up to a sign,
whose elliptic pull-back gives the representation $(A,B)$. 

\begin{cor}\label{Cor:IrredLameArePullBack}
Let $(E,\nabla)$ be an irreducible Lam\'e connection with exponent $\vartheta\not\in2\Z$.
Then there is a unique (up to isomorphism) connection $(\underline E,\underline\nabla)\in \mathcal M_{t}^{\boldsymbol\theta}$
with exponents 
$$\boldsymbol{\theta}=(\frac{1}{2},\frac{1}{2},\frac{1}{2},\frac{1}{2}+\frac{\vartheta}{2})$$
such that $(E,\nabla)$ is the elliptic pull-back of $(\underline E,\underline\nabla)$.
\end{cor}

\begin{proof}Consider $(A,B)$ the monodromy representation of $(E,\nabla)$. 
There is a unique quadruple $(M_0,M_1,M_t,M_\infty)$ lifting to the representation $(A,B)$
such that $\tr(M_\infty)=-2\sin(\frac{\pi\vartheta}{2})$ (one has to conveniently choose the sign of $M$, and therefore of the quadruple). 
Assume non resonance condition $\vartheta\not\in\Z$.
By the Riemann-Hilbert correspondance, there is a unique connection $(\underline E,\underline\nabla)\in \mathcal M_{t}^{\boldsymbol\theta}$
with prescribed monodromy and exponents up to isomorphism. By construction, the elliptic pull-back of $(\underline E,\underline\nabla)$ 
must have exponent $\vartheta$ and holonomy representation $(A,B)$, the same as $(E,\nabla)$. Again by (unicity part of) Riemann-Hilbert correspondance, 
the elliptic pull-back of $(\underline E,\underline\nabla)$ must be isomorphic to $(E,\nabla)$, proving the Corollary in the non resonant case. 
When $\vartheta\in\Z\setminus 2\Z$, the proof is the same if 
the singular point is logarithmic (i.e. with infinite monodromy). However when the pole of $(E,\nabla)$ becomes apparent,
then we have to deal with a parabolic structure to restore injectivity of the Riemann-Hilbert correspondance; we do not detail, but the key step
of the proof is precisely given by Proposition \ref{Prop:LameZapp}.
\end{proof}

\section{Proof of Theorem \ref{T:Main}}\label{sec:proof}

We now detail the proof of Theorem \ref{T:Main}. Let $t\mapsto(E_t,\nabla_t)$ be 
an isomonodromic deformation of an irreducible Lam\'e connection with exponent $\vartheta$.
From Corollary \ref{Cor:IrredLameArePullBack}, it is the elliptic pull-back of an isomonodromic 
deformation $t\mapsto(\underline E_t,\underline\nabla_t)\in\mathcal M^{\boldsymbol\theta}_t$
with $\boldsymbol{\theta}=(\frac{1}{2},\frac{1}{2},\frac{1}{2},\frac{1}{2}+\frac{\vartheta}{2})$.
From Bolibrukh transversality (see section  \ref{sec:Viktoria}), there is an open set of the parameter 
for which the bundle $E_t$ is trivial and the parabolic directions $(l_0,l_1,l_t,l_\infty)$ are pairwise distinct;
moreover, they do not lie on a degree $(1,1)$ curve. Therefore, the cross-ratio
$$c=\frac{l_t-l_0}{l_1-l_0}\frac{l_1-l_\infty}{l_t-l_\infty}\in\P^1\setminus\{0,1,t,\infty\}$$
is not special and we can apply Proposition \ref{P:ElmTransfRuled} and Corollary \ref{cor:TuInvariantGeneric}
and get the explicit expression for Tu invariant
$$\lambda(E_t)=q+\frac{\rho+\kappa_\infty}{p}$$
where $t\mapsto (p(t),q(t))$ are the invariants of $(\underline E_t,\underline\nabla_t)$.
In particular, $\lambda(E_t)$ coincide with the Okamoto symmetric $s_2s_1s_2$ of $q(t)$
(see section \ref{sec:OkaSym}) which is therefore a Painlev\'e VI solution itself, for parameters 
$$\tilde{\boldsymbol\kappa}=\left(\frac{\vartheta}{4},\frac{\vartheta}{4},\frac{\vartheta}{4},\frac{\vartheta}{4}\right).$$

\section{Flat logarithmic $\mathrm{sl}(2,\C)$-connections}
Here, we recall basic facts about flat logarithmic connections
that can be found in greater details in 
\cite{Deligne,NovikovYakovenko,LorayPereira,Heu}.

A {\bf meromorphic connection} of rank $r$ over a smooth complex 
manifold $X$ is a pair $(E,\nabla)$ where $E$ is a locally trivial 
rank $r$ holomorphic vector bundle over $X$ and $\nabla$
a $\C$-linear morphism of sheaves
$$\nabla\ :\ \mathcal E\to \mathcal M(K)\otimes_\mathcal O \mathcal E$$
(where $\mathcal E$ is the sheaf of holomorphic sections of $E$,
and $\mathcal M(K)$, the sheaf of meromorphic sections of the canonical
bundle $K$) satisfying moreover the Leibniz rule
$$\nabla(fv)=df\otimes v+f\nabla(v)$$
for all sections $f$ and $v$ of the structural sheaf $\mathcal O$
and the vector bundle $E$ respectively. 
From the analytic point of view, $E$ is defined
by charts 
$$U_i\times\C^r\ni(x,Y_i),\ \ \ X=\cup_i U_i,$$ 
glued by transition maps 
$$Y_i=M_{i,j}Y_j,\ \ \ M_{i,j}\in\mathrm{GL}(r,\mathcal O(U_i\cup U_j));$$
then $\nabla$ is a differential operator of the form 
$$Y_i\mapsto dY_i-\Omega_i Y_i,\ \ \ \Omega_i\in\mathrm{gl}(r,\mathcal M(K)(U_i)),$$ 
in trivializing charts, satisfying compatibility conditions
$$\Omega_j=M_{i,j}^{-1} \Omega_i M_{i,j} +M_{i,j}^{-1}dM_{i,j}.$$
Throughout this work, we will adopt the later analytic point of view.
Meromorphic connections $(E,\nabla)$ will be considered up to 
holomorphic isomorphisms of vector bundles.

\subsection{Polar divisor}
We say that $\nabla$ has a {\bf pole} at some point 
$x\in U_i\subset X$ 
if at least one of the coefficients 
of the corresponding matrix $\Omega_i$ has a pole at $x$; 
the {\bf order} of the pole is therefore given by the
maximal order for all coefficients. One easily check
that is does not depend neither on the choice of the chart $U_i$,
nor of the local trivialization $Y_i$. The {\bf polar divisor} $D=(\nabla)_\infty$ 
of the connection is a well defined positive divisor on $X$.

\subsection{Flatness and monodromy representation}
A {\bf horizontal section} (or solution) of $(E,\nabla)$
is any section $v$ of $E$ satisfying $\nabla(v)=0$;
in a chart, horizontal sections $Y_i(x)$ are the solutions 
of the Pfaffian system $dY_i=\Omega_iY_i$.
The connection $(E,\nabla)$ is {\bf flat} (or integrable) 
when it satisfies
$$d\Omega_i+\Omega_i\wedge\Omega_i=0$$
in any chart (in once chart is actually enough).
This is equivalent to the existence of a basis
$\mathcal B=(v_1,\ldots,v_r)$ of horizontal holomorphic 
sections at any regular point for $\nabla$. In other
words, the connection is flat if it is locally trivial
at any regular point, i.e. given by $Y_i\mapsto dY_i$ 
($\Omega_i\equiv0$)
in convenient local trivialisation of $E$. This basis $\mathcal B$ admits 
analytic continuation along any paths in $X\setminus D$, just by gluing local
trivializations of $\nabla$ with help of transition maps of the bundle.
Therefore, fixing $x_0\in X\setminus D$ and a basis $\mathcal B$
like above at the neighborhood of $x_0$, we get the 
{\bf monodromy representation} of $(E,\nabla)$ with
respect to $\mathcal B$ a homomorphism
$$\rho_{\nabla,\mathcal B}\ :\ \pi_1(X\setminus D,x_0)\to\mathrm{GL}(r,\C)\ ;\ \gamma\mapsto M_\gamma$$
defined as follows: if $\mathcal B_\gamma$ is the new basis
of horizontal sections around $x_0$ obtained after analytic continuation along $\gamma$, then $M_\gamma$ is given by 
$$\mathcal B_\gamma=\mathcal BM_\gamma.$$
If we change the basis of horizontal sections $\mathcal B$ 
by another one $\mathcal B'=M\mathcal B$, 
$M\in\mathrm{GL}(r,\mathbb  C)$, then the new monodromy
representation is given by
$$\rho_{\nabla,\mathcal B'}(\gamma)=M\cdot\rho_{\nabla,\mathcal B'}(\gamma)\cdot M^{-1},\ \ \ \forall\gamma\in\pi_1(X\setminus D,x_0).$$
Therefore, the monodromy representation is well-defined by $\nabla$
up to $\mathrm{GL}(r,\mathbb  C)$-conjugacy and we will simply
denote by $\rho_\nabla$ any representative of the conjugacy class.

\subsection{Flat logarithmic connections}
A flat connection is said {\bf logarithmic} when it has only simple poles
(i.e. $D$ is reduced) and moreover,  in charts, the 
matrix connection $\Omega$ is such that its differential $d\Omega$
has simple poles as well. This later condition is equivalent 
to the fact that, at the neighborhood any smooth point 
of the polar divisor $D$, the connection has a product structure:
there exists local coordinates $(x_1,\ldots,x_n)$ on $X$ 
and a local trivialization $Y_i$ such that
$D$ is defined by $x_2=\cdots=x_n=0$ and
the matrix connection only depend on the single variable $x:=x_1$.
Therefore, along each irreducible component of $D$,
the $\mathrm{GL}(r,\mathbb  C)$-conjugacy class 
of the residual matrix is constant: one can talk about
the {\bf eigenvalues} $\{\theta_1,\ldots,\theta_r\}$ of the connection 
at each pole, i.e. each component of $D$. 
A pole is said {\bf resonant}
when at least two eigenvalues differ by an integer: 
$\theta_i-\theta_j\in\Z$, $i\not=j$.
At any smooth point of a non resonant pole, the matrix connection 
can be further reduced to its principal part
$$\Omega=\begin{pmatrix}\theta_1&&0\\&\ddots&\\0&&\theta_r\end{pmatrix}\frac{dx}{x}.$$
In the resonant case, for each $\theta_i-\theta_j\in\Z_{\ge 0}$,
the $(i,j)$-coefficient of $\Omega$ can be reduced to a resonant
monomial $c\cdot x^{\theta_i-\theta_j}\cdot\frac{dx}{x}$.
For each irreducible component $D_j$ of the divisor $D$,
fix a path $\delta_j$ in $X\setminus D$ joining the base point $x_0$ 
to a smooth point of $D_j$. Now, consider a loop $\gamma_j$
in $X\setminus D$ based at $x_0$ going first along $\delta_j$
very close to $D_j$, turning once around $D_j$, and going back
to $x_0$ by $\delta_j^{-1}$. The conjugacy class of $\rho(\gamma_j)$
does not depend on the choices and is called the 
{\bf local monodromy} of $(E,\nabla)$ around $D_j$;
eigenvalues are given by $\{e^{2i\pi\theta_1},\ldots,e^{2i\pi\theta_r}\}$.
If the local monodromy is diagonalisable, so is the residual matrix;
the converse is not true. 

\subsection{Trace and twist}
The {\bf trace} of a connection $(E,\nabla)$ is the rank $1$
meromorphic connection $(\det(E),\tr(\nabla))$ where
$\det(E)$ is the determinant of $E$ 
defined (with notations above) by transition charts $\det(M_{i,j})$, 
and $\tr(\nabla)$, defined by $y_i\mapsto dy_i-\tr(\Omega_i) y_i$.
We say that the connection $(E,\nabla)$ is {\bf trace free} 
when its trace is the trivial connection (on the trivial bundle):
$y\mapsto dy$. The polar divisor of the trace is bounded by 
that one of the initial connection. The {\bf twist} of $(E,\nabla)$ 
by a rank $1$ connection $(L,\zeta)$ is the rank $r$ connection
given by their tensor product $(L\otimes E,\zeta\otimes \nabla)$:
if $(L,\zeta)$ is defined in the same
open covering $U_i$ by $y_i\mapsto dy_i-\omega_iy_i$
with transition charts $y_i=m_{i,j}y_j$, then the twist
has the matrix form $\Omega_i+\omega_i\cdot I$
with transition charts $Y_i=m_i\cdot M_i Y_j$.
We have
$$\det(L\otimes E)=\det(L)^{\otimes r}\otimes\det(E)\ \ \ \text{and}\ \ \ 
\tr(\zeta\otimes\nabla)=\tr(\zeta)^{\otimes r}\otimes\tr(\nabla).$$
The trace of a flat (resp. logarithmic) connection is flat 
(resp. logarithmic).

\subsection{$\mathrm{sl}(2,\C)$-connections}
For the sake of notations, we now restrict ourselves to flat logarithmic
$\mathrm{sl}(2,\C)${\bf-connections} (i.e. rank $2$ 
and trace free); the monodromy representation takes values
into $\mathrm{SL}(2,\C)$. For each irreducible component $D_j$
of the polar divisor $D$, the {\bf exponent} $\theta_j\in\C$, defined
up to a sign, is the difference between the two eigenvalues $\pm\frac{\theta_j}{2}$ of the residual matrix: the corresponding local monodromy has trace $2\cos(\pi\theta)$. 
The component $D_j$ is resonant
if, and only if, $\theta_j\in\Z$. In this case, say $\theta_j=n\in\Z_{\le 0}$, the connection matrix
can be reduced to 
$$\text{either}\ \ \ \Omega=\begin{pmatrix}\frac{n}{2}&0\\0&-\frac{n}{2}\end{pmatrix}\frac{dx}{x},\ \ \ \text{or}\ \ \ \Omega=\begin{pmatrix}\frac{n}{2}&x^n\\0&-\frac{n}{2}\end{pmatrix}\frac{dx}{x}$$
at the neighborhood of any smooth point of $D_j$.
The corresponding local monodromy is respectively
$$\pm\begin{pmatrix}1&0\\0&1\end{pmatrix},\ \ \ \text{or}\ \ \ \pm\begin{pmatrix}1&1\\0&1\end{pmatrix}$$
where $\pm:=(-1)^n$.
The pole is called {\bf apparent} in the former case 
(there is no pole when $n=0$) and {\bf logarithmic}
in the later case. In any case, we note that bounded solutions
$$v=\begin{pmatrix}cx^{n/2}\\ 0\end{pmatrix},\ \ \ c\in\C,$$
form a one-dimensional subspace of the space of solutions.

\subsection{Riemann-Hilbert correspondence}
One defines the {\bf Riemann-Hilbert correspondence} as the map
$$\begin{matrix}&RH&\\
\left\{\begin{matrix}\text{Flat logarithmic}\\ \text{$\mathrm{sl}(2,\C)$-connections over $X$}\\ 
\text{with polar divisor $D$}\\
\text{and exponent $\theta_j$ over $D_j$}
\end{matrix}\right\}/\sim     &\to&
\left\{\begin{matrix}\text{Representations}\\ \text{$\pi_1(X\setminus D,x_0)\to\mathrm{SL}(2,\C)$}\\ 
\text{having trace $2\cos(\pi\theta_j)$}\\
\text{around $D_j$}
\end{matrix}\right\}/\sim\\
(E,\nabla)&\mapsto&\rho\end{matrix}$$
which assigns to a connection, up to holomorphic bundle isomorphism,
the corresponding monodromy representation, up to conjugacy.
This map is surjective provided that $D$ has normal crossings
\cite{Deligne} or $X$ has dimension $\le 2$ \cite{LorayPereira}.
It is moreover injective provided that none
of the exponent is a non zero integer. 
In fact, the lack of injectivity comes
from apparent singular points. One can restore the injectivity
in the resonant case by enriching the monodromy data as follows.
For each $\theta_j\in\Z\setminus\{0\}$,
consider, in the basis of solutions $\mathcal B$ near $x_0$,
the one-dimensional subspace $L_j\subset\C^2$ of those solutions 
that are bounded around $D_j$ after analytic continuation 
along the path $\delta_j$. The full monodromy data, characterizing
the connection up to isomorphism, is now given by 
the monodromy representation $\rho$ and the collection
$L_j\in\P^1$($=\P(\C^2)$) 
where $j$ spans over the set $J^{res}$ of all indices such that
$\theta_j\in\Z\setminus\{0\}$. Any base change 
$\mathcal B'=M\mathcal B$, $M\in\mathrm{SL}(2,\C)$, yields new monodromy data
\begin{equation}\label{Eq:BaseChange}
\rho'=M\cdot\rho\cdot M^{-1}\ \ \ \text{and}\ \ \ L_j'=M\cdot L_j,\ \forall j\in J^{res}.
\end{equation}
(for the standart action of $\mathrm{SL}(2,\C)$
on $\C^2$).

\begin{prop}\label{prop:RHparabolic}
Assume $D$ is normal crossing reduced divisor,
and $D_j$, $\gamma_j$, $\delta_j$ be as above.
The set of flat logarithmic $\mathrm{sl}(2,\C)$-connections 
$(E,\nabla)$ with polar divisor $D$ and exponents $\theta_j$ along $D_j$
modulo isomorphism is in one-to-one correspondence 
with the set of pairs
$\left( \rho, (L_j)_{j\in J^{res}} \right)$
where
\begin{itemize}
\item $\rho\in\mathrm{Hom}(\pi_1(X\setminus D,x_0),\mathrm{SL}(2,\C))$ such that $\tr(\rho(\gamma_j))=2\cos(\pi\theta_j)$ 
for all $j\in J$,
\item $L_j\in\P^1$ is $\rho(\gamma_j)$-invariant 
for all $j\in J^{res}$
\end{itemize}
modulo the $\mathrm{SL}(2,\C)$-action given by
(\ref{Eq:BaseChange}).
\end{prop}

\subsection{Reducible $\mathrm{gl}(2,\C)$-connections}
A sub line bundle $L\subset E$ is said $\nabla$-invariant
when it is generated by $\nabla$-horizontal sections.
In this case, the connection $\nabla$ induces a meromorphic
connection $\nabla\vert_L$ on $L$. 
The connection $(E,\nabla)$ is said {\bf reducible}
when it admits such an invariant line bundle, 
and {\bf irreducible} if not. When the connection is
reducible, then the monodromy representation is
itself reducible: the monodromy group has a common
eigenvector. In the logarithmic case with normal crossing
polar divisor, the converse is true: $(E,\nabla)$ is reducible
if, and only if, $\rho$ is. 

\subsection{Projective $\mathrm{sl}(2,\C)$-connections and Riccati foliation}\label{sec:ProjConRiccati}
A rank $2$ meromorphic connection $(E,\nabla)$ 
induces a projective $\mathrm{sl}(2,\C)$-connection
$(\P(E),\P(\nabla))$ on $X$.
If the linear connection is given in the trivializing chart $Y_i$ by
$$Y_i\mapsto dY_i-\Omega_iY_i,\ \ \ \Omega_i=\begin{pmatrix}\alpha_i& \beta_i\\
\gamma_i& \delta_i\end{pmatrix},$$
where $\alpha_i$, $\beta_i$, $\gamma_i$ and $\delta_i$ are meromorphic $1$-forms on $U_i$,
then the projective connection $\P(\nabla)$ is defined 
in the projective trivializing chart 
$\P(Y_i)=(1:z_i)\in\P^1$ by 
$$z_i\mapsto dz_i+\beta_i z_i^2+(\alpha_i-\delta_i)z_i-\gamma_i.$$
Another linear connection $(E',\nabla')$ will define the same
projective connection if and only if it is the twist of
$(E,\nabla)$ by some rank $1$ meromorphic connection $(L,\zeta)$.

Conversely, when $\mathrm{H}^2(X,\mathcal O^*)=0$ 
(as it so happens when $X$ is a curve), 
any $\P^1$-bundle $P$ is 
the projectivization of a rank $2$ vector bundle $P=\P(E)$. 
Moreover, one immediately deduce from the formula above that
given any meromorphic projective $\mathrm{sl}(2,\C)$-connection on $P$ and any meromorphic (linear) connection
$\zeta$ on the line bundle $L=\det(E)$, there is a unique
meromorphic linear connection $\nabla$ on $E$ lifting
the projective one on $\P(E)$ with prescribed
trace $\tr(\nabla)=\zeta$ on $\det(E)$. 

When $X$ is a curve,
there are two topological types of $\P^1$-bundles~:
the topological type of $\P(E)$ is given by the class 
of $\deg(E)\in\Z/2\Z$. We note that
topological triviality is the condition for the existence of a square root
$L$ of $\det(E)\in\mathrm{Pic}(X)$.
In other words, a $\P^1$-bundle $P$ is topologically trivial
if, and only if, $P$ can be lifted as an $\mathrm{SL}(2,\C)$-vector bundle : setting $E:=E\otimes L^{\otimes(-2)}$, we get
$P=\P(E)$ with $\det(E)=\mathcal O$. 
The $\mathrm{SL}(2,\C)$-lifting depends on the choice 
of a square root : it is well defined up to order two points in 
$\mathrm{Pic}(X)$ and there are $2^{2g}$ possible liftings
over $X$ of genus $g$. Finally, any meromorphic projective
connection on $P$ lifts uniquely as a linear $\mathrm{sl}(2,\C)$-connection $\nabla$ on $E$ (with the same polar divisor).

When $X$ is a curve, the total space of $\P(E)$ 
is a ruled surface $S\to X$ and the Riccati equation 
$\P(\nabla)=0$ defines a singular foliation $\mathcal F$ 
on $S$ whose leaves are the graphs of horizontal sections
of the projective connection. The pair $(S,\mathcal F)$ is called
a {\bf Riccati foliation} (see \cite{Brunella}).  The foliation is
regular, transversal to the ruling outside the polar locus
of $\nabla$. Over the poles of the projective connection,
the $\P^1$ fibre is the disjoint union of a vertical leaf
of $\mathcal F$ and $1$ or $2$ singular points. 
Precisely, when $\nabla$ is logarithmic (simple poles),
singular points correspond to the eigenlines
of the linear connection $\nabla$. If $\theta$ and $\theta'$
are the eigenvalues of $\nabla$ at some pole $x\in X$,
denote by $l$ and $l'$ the corresponding eigenlines;
the exponent (or Camacho-Sad index of the vertical leaf) at the singular point $l$ of the Riccati foliation is 
$\kappa=\theta'-\theta$. Let $\sigma:X\to S$ be a section.

\begin{prop}\label{prop:CamachoSad}
Let $X$ be a curve, $(S,\mathcal F)$ a Riccati foliation
with simple poles over $X$, and $\sigma$ a $\mathcal F$-invariant
section. Then 
$$\sigma\cdot\sigma=\sum_{i}\kappa_i$$
where $\kappa_i$ are the exponents
of the singular points $\sigma$ passes through 
where $i$ runs over the invariant fibres of $\mathcal F$.
\end{prop}

This is a particular case of Camacho-Sad
formula (see \cite{Brunella}, page 37).

\begin{proof}Viewed as a projective connection, there exists 
a unique lifting $(E,\nabla)$ of the projective connection
such that the $\nabla$-invariant line bundle $L$ corresponding to 
$\sigma$ is the trivial bundle, and the connection induced by
$\nabla$ on $L$ is the trivial connection: the eigenvalues of 
$\nabla$ over the pole $i$ are given by $0$ and $\kappa_i$.
Then, Fuchs relations give
$$\deg(E)=\sum_{i}\kappa_i$$
and we have
$$\sigma\cdot\sigma=\deg(E)-2\deg(L)=\deg(E).$$
\end{proof}

\begin{prop}\label{prop:BrunellaSection}
Let $X$ be a curve of genus $g$, $(S,\mathcal F)$ a Riccati foliation
over $X$ with $n$ poles (counted with multiplicity), 
and $\sigma:X\to S$ a section which is not $\mathcal F$-invariant. Then the number of tangencies
between $\sigma$ and the foliation (including singular points lying
on $\sigma$) is given by
$$\mathrm{tang}(\mathcal F,\sigma)=2g-2+n+\sigma\cdot\sigma$$
(counted with multiplicities).
\end{prop}

This is a particular case of Proposition 2, page 37 in \cite{Brunella}.

\begin{proof}Choose any lifting $(E,\nabla)$ of the projective connection and apply Lemma \ref{lem:Brunella} to the line bundle
$L\subset E$ corresponding to $\sigma$.
\end{proof}

\subsection{Stability of bundle and connections}
A rank $2$ vector bundle over a curve $X$ is said
stable (resp. semistable) when we have
\begin{equation}
\det(E)-2\det(L)>0\ \ \ (\text{resp.}\ \ge0)
\end{equation}
for all sub line bundle $L\subset E$. 
This notion is invariant by projective equivalence:
the $\P^1$-bundle $\P(E)$
is stable (resp. semistable) if
$$\sigma\cdot\sigma>0\ \ \ (\text{resp.}\ \ge0)$$
for all section $\sigma:X\to \P(E)$.

Similarly, we say that a connection $(E,\nabla)$ is stable (resp. semistable)
when 
\begin{equation}
\det(E)-2\det(L)>0\ \ \ (\text{resp.}\ \ge0)
\end{equation}
for all $\nabla$-invariant sub line bundle $L\subset E$.
Again, this notion is invariant by projective equivalence:
the projective connection $(\P(E),\P(\nabla))$
is stable (resp. semistable) if 
$$\sigma\cdot\sigma>0\ \ \ (\text{resp.}\ \ge0)$$
for all $\P(\nabla)$-invariant section $\sigma$.
In particular, an irreducible connection $(E,\nabla)$
is stable even if the bundle $E$ is unstable. However, 
for a semistable connection, the bundle $E$ cannot be arbitrarily
unstable: by Lemma \ref{lem:Brunella}, the stability index
is bounded by
$$\det(E)-2\det(L)\ge 2-2g-\deg(D)$$
where $D$ is the polar divisor of $\nabla$.

\subsection{Meromorphic and elementary gauge tranformations}\label{ss:elm}

We start recalling what is an {\bf elementary transformation}
of a rank $2$ vector bundle, say $E$, over a curve $X$.
Given a point $p\in X$ and a linear subspace $l\in\P(E_p)$
in the fibre over $p$, one usually defines two birational bundle transformations
$$\mathrm{elm}_{p,l}^+:E\dashrightarrow E^+\ \ \ \text{and}\ \ \ 
\mathrm{elm}_{p,l}^-:E\dashrightarrow E^-,$$
that are unique up to post-composition by a bundle isomorphism.
In restriction to the punctured curve $X^*=X\setminus\{p\}$, 
both $\mathrm{elm}^\pm_{p,l}$ induce isomorphisms.
At the neighborhood of $p$, they can be described as follows.
Choose a local coordinate $x:U\to\C$ at $p$ 
together with a trivialization of $Y:E\vert_U\to \C^2$
for which the linear subspace $l$ is spanned by $Y=\begin{pmatrix}0\\ 1\end{pmatrix}$.
This, in particular, induces a trivialization of $E\vert_{X^*}$
on $U^*=U\setminus\{p\}$. Elementary transformations 
$\mathrm{elm}^\pm_{p,l}$ can be
defined by the following commutative diagram
$$ \xymatrix{
E\ar@{-->}[d]_{\mathrm{elm}^\pm_{p,l}}
& E\vert_{X^*}\ar@{=}[d]^{\mathrm{id}} 
& E\vert_{U^*}\ar@{_{(}->}[l]\ar@{=}[d]^{\mathrm{id}}\ar[r]^-Y_-\sim 
& U^*\times\C^2\ar@{=}[d]^{\mathrm{id}}\ar@{^{(}->}[r]^{\mathrm{id}}
& U\times\C^2\ar@{-->}[d]^{\phi^{\pm}} \\
E^{\pm} & E\vert_{X^*} 
& E\vert_{U^*}\ar@{_{(}->}[l]\ar[r]^-Y_-\sim  
& U^*\times\C^2\ar@{^{(}->}[r]^{\phi^{\pm}}
& U\times\C^2
}$$
where 
$$\phi^+(Y)=\begin{pmatrix}1&0\\0&x\end{pmatrix}Y\ \ \ \text{and}\ \ \ 
\phi^-(Y)=\begin{pmatrix}1/x&0\\0&1\end{pmatrix}Y.$$
All three bundles $E$ and $E^\pm$ are constructed by gluing the local trivial
bundle $U\times\C^2$ to the same restricted bundle 
$E\vert_{X^*}$
through different bundle isomorphisms (either the identity, or $\phi^\pm$) over the punctured neighborhood $U^*$. Isomorphisms 
$E\vert_{X^*}\to E^\pm\vert_{X^*}$
given by this construction extend as birational bundle transformations.
We have
$$\det(E^\pm)=\det(E)\otimes\mathcal O(\pm[p]).$$
On the other hand, $\mathrm{elm}^\pm_{p,l}$ induce 
the same birational $\P^1$-bundle transformation
$$\mathrm{elm}_{p,l}:P=\P(E)\dashrightarrow P'$$
since $\phi^+$ and $\phi^-$ coincide both in $\mathrm{PGL}(2,\mathcal O(U^*))$ and $\mathrm{PGL}(2,\mathcal M(U))$. 

One still has to verify that our construction only depends
on the ``parabolic structure'' $(p,l)$, not on the choice
of the local trivialization $Y$. For another choice
$$\tilde Y=M\cdot Y,\ \ \ M\in\mathrm{GL}(2,\mathcal O(U)),$$
one has to check that
$\phi^+(\tilde Y)=\phi^+(M\cdot Y)=\tilde M\cdot \phi^+(Y)$
with $\tilde M\in\mathrm{GL}(2,\mathcal O(U))$.
Indeed, if $M=\begin{pmatrix}a&b\\ c&d\end{pmatrix}$, 
then $\tilde M=\begin{pmatrix}a&b/x\\ xc&d\end{pmatrix}$;
since $l$ has to be spanned by 
$\tilde Y=\begin{pmatrix}0\\ 1\end{pmatrix}$,
we have $b(0)=0$ and $\tilde M$ is holomorphic
with $\det(\tilde M)=\det(M)\not=0$.

A similar computation shows that the line $l^\pm\subset E^\pm_p$, 
defined by $\phi^\pm=\begin{pmatrix}1 \\ 0\end{pmatrix}$ in the
construction above, does not depend on our choices.
In other word, given a bundle $E$ equipped with a parabolic
structure over $p$, $l\subset E_p$, elementary transformations
define a birational tranformation
$$\mathrm{elm}_p^\pm:(E,l)\dashrightarrow(E^\pm,l^\pm)$$
between parabolic bundles which is well defined up to 
left-and-right composition by parabolic bundle isomorphisms.
It also follows from computations above that 
$$\mathrm{elm}_p^\pm\circ\mathrm{elm}_p^\mp:(E,l)\to(E',l')$$ 
are parabolic bundle isomorphisms. In this sense, 
$\mathrm{elm}_p^+$ and $\mathrm{elm}_p^-$ 
are inverse to each other. We can 
also consider a general rank $2$ parabolic bundle
$(E,l)$ over $(X,S)$ where $S\subset X$ is a finite subset,
and $l:S\to\P(E\vert_S)$ a section of the projective 
bundle induced over $S$. The elementary transformations
$\mathrm{elm}_p^\pm:(E,l)\dashrightarrow(E^\pm,l^\pm)$
are defined between parabolic bundles over $(X,S)$ like above 
when $p\in S$ (note that $\mathrm{elm}_{p,l(p)}^\pm$
induces an isomorphism of parabolic bundles over $(X^*,S^*)$)
and as the identity when $p\not\in S$. Finally, if $p_1,p_2\in S$
are two distinct points, elementary transformations 
$\mathrm{elm}_{p_1}^\pm$ and $\mathrm{elm}_{p_2}^\pm$
commute (up to parabolic bundle isomorphisms) so that one
can define $\mathrm{elm}_{S'}^\pm$ for any subset $S'\subset S$.

Now, we would like to describe how elementary transformations
act on parabolic connections $(E,\nabla,l)$: $(E,l)$ is a parabolic
bundle over $(X,S)$ like above and $\nabla$ a meromorphic connection on $E$. Let $p\in S$ and denote by $\nabla^\pm$
the push-forward of $\nabla$ by the elementary transformation
$\mathrm{elm}_p^\pm:(E,l)\dashrightarrow(E^\pm,l^\pm)$:
$\nabla^\pm$ is a meromorphic connection on $E^\pm$.
Under notations above, if $\nabla$ is defined in coordinates $(x,Y)$
by 
$$Y\mapsto dY-\Omega Y,\ \ \ \Omega=\begin{pmatrix}\alpha&\beta\\ \gamma&\delta\end{pmatrix},$$
then $\nabla^\pm$ is defined in the $E^\pm$ local trivialization 
$Y^\pm=\phi^\pm(Y)$ by $Y^\pm\mapsto dY^\pm-\Omega^\pm Y\pm$
where 
$$\Omega^+=\begin{pmatrix}\alpha&\frac{\beta}{x}\\ x\gamma&\delta+\frac{dx}{x}\end{pmatrix}\ \ \ \text{and}\ \ \ 
\Omega^-=\begin{pmatrix}\alpha-\frac{dx}{x}&\frac{\beta}{x}\\ x\gamma&\delta\end{pmatrix}.$$
If $p$ is not a pole of $\nabla$, then $\nabla^\pm$ has a logarithmic
pole at $p$. If $p$ is a pole of order $k$ of $\nabla$, there are two cases:
\begin{itemize}
\item if $l(p)$ is an eigenvector of $\nabla$ at $p$ 
(i.e. of $x^k\Omega$ at $x=0$), then $\nabla^\pm$ 
has a pole of order $k$ or $k-1$ at $p$;
\item if not, then $\nabla^\pm$ 
has a pole of order $k+1$ at $p$.
\end{itemize}
In any case, $l^\pm(p)$ is an eigenvector of $\nabla^\pm$ at $p$.

When $\nabla$ is a logarithmic connection, then $\nabla^\pm$
is also logarithmic if, and only if, either $p$ is regular, or $p$ is a pole
and $l(p)$, an eigenline of $\nabla$. One
can then choose the coordinate $Y$ such that $\nabla$ is defined
by $Y\mapsto dY-AY\frac{dx}{x}$ with
$$\Omega=\begin{pmatrix}\theta_1\frac{dx}{x}&0\\0&\theta_2\frac{dx}{x}\end{pmatrix},\ \ \ 
\begin{pmatrix}(\theta+n)\frac{dx}{x}&x^n\frac{dx}{x}\\ 0&\theta\frac{dx}{x}\end{pmatrix}\ \ \ \text{or}\ \ \ 
\begin{pmatrix}\theta\frac{dx}{x}&0\\ x^n\frac{dx}{x}&(\theta+n)\frac{dx}{x}\end{pmatrix}$$
$$A=\begin{pmatrix}\theta_1&0\\0&\theta_2\end{pmatrix},\ \ \ 
\begin{pmatrix}(\theta+n)&x^n\\ 0&\theta\end{pmatrix}\ \ \ \text{or}\ \ \ 
\begin{pmatrix}\theta&0\\ x^n&(\theta+n)\end{pmatrix}.$$
(including the regular case $A=0$) with the restriction $n>0$ in the middle case. Then $\nabla^\pm$
is given in coordinate $Y^\pm=\phi^\pm(Y)$ by $Y^\pm\mapsto dY^\pm-A^\pm Y^\pm\frac{dx}{x}$ with respectively
$$A^+=\begin{pmatrix}\theta_1&0\\0&\theta_2+1\end{pmatrix},\ \ \ 
\begin{pmatrix}(\theta+n)&x^{n-1}\\ 0&\theta+1\end{pmatrix}\ \ \ \text{or}\ \ \ 
\begin{pmatrix}\theta&0\\ x^{n+1}&(\theta+n+1)\end{pmatrix}$$
and
$$A^-=\begin{pmatrix}\theta_1-1&0\\0&\theta_2\end{pmatrix},\ \ \ 
\begin{pmatrix}(\theta+n-1)&x^{n-1}\\ 0&\theta\end{pmatrix}\ \ \ \text{or}\ \ \ 
\begin{pmatrix}\theta-1&0\\ x^{n+1}&(\theta+n)\end{pmatrix}.$$
We resume:
\begin{itemize}
\item if $p$ is a regular point of $\nabla$, i.e. $A=0$, 
then $\nabla^\pm$ is logarithmic with eigenvalues $0$ and $\pm1$;
\item if $p$ is a pole and $\theta$ an eigenvalue of $\nabla$ at $p$,
there exists one and only one eigenline $l(p)$ associate to 
$\theta$ except in the diagonal case when $\theta_1=\theta_2=\theta$;
when $\nabla$ is trace free, this later case does not occur 
and (the eigenspace of) each eigenvalue defines a parabolic structure
over $p$;
\item if $\{\theta_1,\theta_2\}$ denote the eigenvalues and if
$l(p)$ is the eigenline associate to $\theta_2$,
then $\nabla^+$ (resp. $\nabla^-$) has eigenvalues 
$\{\theta_1,\theta_2+1\}$ (resp. $\{\theta_1-1,\theta_2\}$);
$\nabla^\pm$ are of diagonal type if, and only if, $\nabla$ is;
the parabolic structure of $(E^\pm,\nabla^\pm,l^\pm)$ over $p$ corresponds to the eigenvalue $\theta_1$.
\end{itemize}
The trace of the connection is changed by
$$\tr(\nabla^\pm)=\tr(\nabla)\otimes\zeta^\pm$$
where $\zeta$ is the unique logarithmic connection 
on $\mathcal O_X(\pm[p])$ having a single pole at $p$
with residue $\pm1$ and trivial monodromy.
Indeed, the monodromy does not change by a birational
bundle transformation.

\end{document}